\documentclass[12 pt]{amsart}

\usepackage{fullpage} 
\usepackage{diagbox}

\usepackage{hyperref}
\usepackage{etex}
\usepackage[shortlabels]{enumitem}
\usepackage{amsmath}
\usepackage{amsxtra}
\usepackage{amscd}
\usepackage{amsthm}
\usepackage{amsfonts}
\usepackage{amssymb}
\usepackage{eucal}
\usepackage{cleveref}
\usepackage[all]{xy}
\usepackage{graphicx}
\usepackage{tikz}
\usetikzlibrary{cd}
\usetikzlibrary{fit, patterns}
\usepackage{mathrsfs}
\usepackage{subfiles}
\usepackage{mathpazo}
\usepackage[colorinlistoftodos, textsize=tiny]{todonotes}
\usepackage{morefloats}
\usepackage{pdfpages}
\usepackage{thm-restate}
\usepackage[utf8]{inputenc}
\usepackage[T1]{fontenc}
\usepackage{epigraph}
\usepackage{csquotes}
\usepackage{moreverb}
\usepackage[margin=1.5in]{geometry}
\usepackage{adjustbox}
\usepackage{accents}
\usepackage{xparse}

\graphicspath{ {images/} }

\RequirePackage{color}
\definecolor{myred}{rgb}{0.75,0,0}
\definecolor{mygreen}{rgb}{0,0.5,0}
\definecolor{myblue}{rgb}{0,0,0.65}

\usepackage{hyperref}
\hypersetup{citecolor=blue}
\usepackage{tikz}
\usetikzlibrary{matrix,arrows,decorations.pathmorphing}

\theoremstyle{plain}
\newtheorem{theorem}[subsubsection]{Theorem}

\newtheorem{proposition}[theorem]{Proposition}
\newtheorem{lemma}[theorem]{Lemma}

\theoremstyle{definition}
\newtheorem{definition}[theorem]{Definition}
\newtheorem{remark}[theorem]{Remark}
\newtheorem{example}[theorem]{Example}

\newtheorem{notation}[theorem]{Notation}

\theoremstyle{remark}
\numberwithin{equation}{section}

\newcommand\nc{\newcommand}
\nc\on{\operatorname}
\nc\renc{\renewcommand}

\newcommand*{\shom}{\mathscr{H}\kern -.5pt om}
\newcommand*{\stor}{\mathscr{T}\kern -.5pt or}
\newcommand*{\sext}{\mathscr{E}\kern -.5pt xt}

\makeatletter
\providecommand\@dotsep{5}
\renewcommand{\listoftodos}[1][\@todonotes@todolistname]{%
\@starttoc{tdo}{#1}}
\makeatother

\makeatletter
\def\Ddots{\mathinner{\mkern1mu\raise\p@
		\vbox{\kern7\p@\hbox{.}}\mkern2mu
		\raise4\p@\hbox{.}\mkern2mu\raise7\p@\hbox{.}\mkern1mu}}
\makeatother

\newcommand{\customlabel}[2]{\protected@write \@auxout {}{\string \newlabel {#1}{{#2}{\thepage}{#2}{#1}{}} }\hypertarget{#1}{#2}}

\DeclareMathOperator\id{id}

\DeclareMathOperator\sgn{sgn}

\DeclareMathOperator\Hur{Hur}
\DeclareMathOperator\Conf{Conf}

\DeclareMathOperator\conf{Conf}

\DeclareFontFamily{U}{wncy}{}
\DeclareFontShape{U}{wncy}{m}{n}{<->wncyr10}{}
\DeclareSymbolFont{mcy}{U}{wncy}{m}{n}
\DeclareMathSymbol{\Sha}{\mathord}{mcy}{"58}

\newcommand{\phur}[4]{\operatorname{Hur}^{#1, #3}_{#2,#4}}

\newcommand{\boundarycphurc}[4]{\operatorname{CHur}^{#1, #3, #4}_{#2}}

\newcommand{\phurc}[3]{\operatorname{Hur}^{#1, #3}_{#2}}

\newcommand{\cphurc}[3]{\operatorname{CHur}^{#1, #3}_{#2}}

\newcommand{\fncell}[4]{\operatorname{W}^{#1, #3, #4}_{#2}}
\newcommand{\cfncell}[4]{\operatorname{CW}^{#1, #3, #4}_{#2}}
\newcommand{\fncocycle}[4]{\operatorname{Z}^{#1, #3, #4}_{#2}}

\newcommand{\fncoboundary}[4]{\operatorname{B}^{#1, #3, #4}_{#2}}

\newcommand{\fncohomology}[4]{\operatorname{H}^{#1, #3, #4}_{#2}}

\newcommand\stable[4]{H_{#1, #3}^{#2, #4}}
\newcommand\stableboundary[4]{H_{#1}^{#2,#3,#4}} 

\def\listtodoname{List of Todos}
\def\listoftodos{\@starttoc{tdo}\listtodoname}

\title{An alternate computation of the stable homology of dihedral group
Hurwitz spaces}

\author{Aaron Landesman}
\author{Ishan Levy}

\usepackage{microtype}
\begin{document}

\begin{abstract}
	We give an different proof of our result computing the stable
	homology of dihedral group Hurwitz spaces. This proof employs more elementary
	methods, instead of higher algebra.
\end{abstract}
\date{October 29, 2024}
\maketitle
\tableofcontents

\section{Introduction}

Let $G$ be a group, $c \subset G$ be a conjugacy class, and $n \in \mathbb
Z_{\geq 0}$. Let $\phurc G n c$ denote the Hurwitz space over the complex
numbers which 
parameterizes branched $G$-covers of the disk, with a marked basepoint over the
boundary, branched at $n$ points, where the inertia type of each branch point
lies in the conjugacy class $c \subset G$. 
In algebraic topology,
these Hurwitz spaces can be described as the homotopy quotient $c^n/B_n$, where $B_n$ is the braid
group on $n$ strands.

Let $\conf_n$ denote the configuration space of $n$
unordered, distinct points in the interior of the disc, and unbranched over the
boundary.
There is a natural map $\phurc G n c \to \conf_n$ which
sends a branched cover of the disc $X \to D$ to its branch locus in the interior
of the disc.

\begin{theorem}
	\label{theorem:stable-homology-intro}
Choose an odd prime $\ell$ and use $G$ to denote the dihedral group of order
$2\ell$, $G : = \mathbb Z/\ell \mathbb Z \rtimes \mathbb
Z/2 \mathbb Z$.
Let $c \subset G$ denote the conjugacy class of order $2$ elements.
There are constants $I$ and $J$ depending only on $G$ so that for 
	$n > iI + J$ and any connected component $Z \subset \phurc G n c$, the map $H_i(Z, \mathbb Q) \to H_i(\conf_n, \mathbb Q)$
	is an isomorphism.
\end{theorem}
This is an immediate consequence of
\cite[Theorem 6.1]{EllenbergVW:cohenLenstra}, which proves the homology
stabilizes in a linear range, and
\autoref{theorem:stable-homology}, which computes the stable value of these
homology groups.

Our primary motivation for proving \autoref{theorem:stable-homology-intro}
is for its application to the computation of the moments predicted by the
number theoretic 
Cohen--Lenstra heuristics over function fields.
See \cite{landesmanL:the-stable-homology-of-non-splitting}
for more on this application.

\begin{remark}
	\label{remark:}
	We originally came up with a more elementary argument for
	\autoref{theorem:stable-homology-intro} along the lines in this paper, which
	guided us to later devise
	the argument presented in
	\cite{landesmanL:the-stable-homology-of-non-splitting}.
	In some sense, these two arguments follow very similar trajectories,
	although this may not be evident on first blush.
	Our main result \autoref{theorem:stable-homology-intro} is a special case of the results of
\cite{landesmanL:the-stable-homology-of-non-splitting}, where we study the stable
homology of $G = \mathbb Z/\ell \mathbb Z \rtimes \mathbb Z/2 \mathbb Z$,
instead of the more general case that $G$ is a finite group and $c \subset G$ is
a conjugacy class so that $(G,c)$ is {\em non-splitting}, i.e., $(G,c)$ has the property that $c \cap G'$ does not consist of more
than one conjugacy class for any subgroup $G' \subset G$.

	We chose to write that version there, because it seemed conceptually
	simpler and also generalized more easily to the case when
	$(G,c)$ is non-splitting.
	However, we still thought it would be nice to record an argument along the lines of our original
	one, especially for those interested in the subject and 
	not familiar with higher algebra. 

	We also wanted to write this note to emphasize that our computation of
	stable homology is not an abstruse result in topology.
	If readers unfamiliar with
	higher algebra only read
	\cite{landesmanL:the-stable-homology-of-non-splitting}, it is possible they
	could view it that way. We wrote this exposition to emphasize
	that our computation of the stable homology of Hurwitz spaces is an accessible result in linear
	algebra.
\end{remark}
\begin{remark}
	\label{remark:}
	With significant additional work, we were able to generalize many parts
	of the argument presented here to the case that
$(G,c)$ is non-splitting. However, we did not carefully work
	out all the details. In particular, we did not work out the generalization of
	\autoref{section:stabilized-by-one-element}.
	We believe it would be interesting to do so.
\end{remark}

\subsection{Proof outline}
\label{subsection:proof-outline}

To prove \autoref{theorem:stable-homology-intro},
we first study what the stabilization map looks like more concretely.
By combining various ideas in the literature, namely the group
completion theorem and the interaction of the stabilization map with boundary
monodromy, we are able to show that the homology in the kernel of the map to
configuration space decomposes as a direct sum of
subspaces, each of which is stabilized via multiplication by a single element of
the group. 
This reduction is carried out in \autoref{section:stable-homology}.
The crux of the matter is therefore to prove
\autoref{proposition:stabilize-by-one-element-trivial}, showing that any element
of the stable cohomology
of $\phurc G n c$ not pulled back from the cohomology of configuration
space, stabilized via
multiplication by a single $g \in c$,
is trivial. 
To approach this, we start by studying the stable cohomology in terms of
Fox-Neuwirth/Fuks cells.
Since we know the stabilization map is obtained via a single group element, we
can write down an explicit form for such a cohomology class stabilized by a
single group element.
If we are able to show this class is cohomologous to one in a sufficiently simple form, we will be
able to conclude using exactness of the stable $\mathcal K$-complex, a complex introduced
in
\cite[Theorem 4.2]{EllenbergVW:cohenLenstra}, which was the key to showing that the
homologies of these Hurwitz spaces stabilize.
In order to massage the element into the desired form,
we extend the study of the $\mathcal K$-complex and introduce a new object of study, which we call the ``two-sided
$\mathcal K$-complex,'' generalizing the $\mathcal K$-complex.
We compute the cohomology of a variant of a particular two-sided $\mathcal K$-complex 
in \autoref{proposition:upper-triangle-exact},
via an
explicit chain homotopy. This allows us to massage $x$ into the desired form,
and hence conclude the proof.

\subsection{Outline of paper}
We introduce notation for Hurwitz spaces in
\autoref{section:hurwitz-space-notation}.
In \autoref{section:2-sided}, we compute the homology of a particular complex,
which we call a two-sided $\mathcal K$-complex.
In \autoref{section:stabilized-by-one-element}, we use this computation of the
homology of the two-sided $\mathcal K$-complex to compute the part of the stable homology of 
dihedral Hurwitz spaces stabilized by a single element $g \in c$.
Finally, in \autoref{section:stable-homology}, we use this computation to prove
\autoref{theorem:stable-homology-intro}, computing the stable homology of
dihedral group Hurwitz spaces.

\subsection{Acknowledgements}

We are grateful to Jordan Ellenberg for meticulously reading and
thinking through some of the most technical parts of the argument, which led to
a substantial improvement in the quality of the writing of this paper.
We thank Melanie Wood for listening to a detailed explanation of our argument.
We thank Jordan Ellenberg and Craig Westerland for sharing with us a number of
the ideas they tried while working on this problem.
We thank Nathalie Wahl and Melanie Wood for helpful suggestions which led to a significant
simplification of the proof of \autoref{theorem:two-sided-g}.
We also thank Andrea Bianchi for pointing out another chain homotopy which can
be used, see \autoref{remark:alternate-homotopy}.
We also thank Ahn
Hoang for helpful comments.
Landesman 
was supported by the National Science
Foundation 
under Award No.
DMS 2102955.
Levy was supported by the NSF Graduate Research
Fellowship under Grant No. 1745302, and by the Clay Research Fellowship.

\section{Notation for Hurwitz spaces}
\label{section:hurwitz-space-notation}

\begin{notation}
	\label{notation:group}
	We let $G$ denote the dihedral group of order $2\ell$ for an odd prime $\ell$ and let $c \subset
	G$ denote the conjugacy class of order $2$ elements.
	As mentioned in the introduction, we use
	$\phurc G n c,$ to denote the Hurwitz space over the complex
numbers which 
branched $G$-covers of the disk, with a marked basepoint over the boundary,
branched at $n$ points, where the inertia type of each branch point lies in the conjugacy class $c \subset G$. 
This has a model as the homotopy quotient $c^n/B_n$, where $B_n$ is generated by
the standard elements $\sigma_1, \ldots, \sigma_{n-1}$ where $\sigma_i$ twists strands $i$ and
$i+1$,
and the action is given by $\sigma_i(g_1, \ldots, g_n) = (g_1, \ldots, g_{i-1}, g_{i+1}, g_{i+1}^{-1} g_i g_{i+1},
g_{i+2},\ldots, g_n)$.

We also define $\cphurc G n c$ to be the union of connected components of
$\phurc G n c$ which parameterize connected $G$ covers of the disk.
If $S \subset c^n$ is the subset of tuples of $n$ elements in $c$ which generate
$G$, then $\phurc G n c$ has a model as the homotopy quotient $S/B_n$.
\end{notation}

\begin{notation}
	\label{notation:component}
	Continuing with notation as in \autoref{notation:group},
	we can identify the connected components of the pointed Hurwitz space
	$\phur G n c {\mathbb C}$ with orbits of the $B_n$ action on $c^n$.
	Under this identification,
	$x_1, \ldots, x_n \in c$ we use $[x_1] \cdots [x_n]$ to denote the
	connected component of $\phur G n c {\mathbb C}$ corresponding to the
	$B_n$ orbit of the tuple $(x_1, \ldots, x_n)$.
\end{notation}

\begin{notation}
	\label{notation:stable-homology}
	Fix $G,c$ as in \autoref{notation:group}.
	For $\beta \in \{0,1\}$,
	we let $\stable i G \beta c$ denote the vector space $H_i(\cphurc G n c,
	\mathbb Q)$ for
	$n \equiv \beta \pmod 2$ sufficiently large. We note this vector space is independent of $n$ once $n$ is
	sufficiently large by 
	\cite[Theorem 6.1]{EllenbergVW:cohenLenstra} (with input from
	\cite[Proposition
A.3.1]{ellenbergL:homological-stability-for-generalized-hurwitz-spaces}).
	Here, the isomorphism between 
	$H_i(\cphurc G n c, \mathbb Q)$ and $H_{i+2}(\cphurc G n c,
	\mathbb Q)$ is given by 
	the stabilization operator $\sum_{g \in c} [g]^2$, where
	$[g]$ corresponds to right multiplication by a generator of $H_0(\phurc
	G 1 c, \mathbb Q)$
	with monodromy $g$.
	At some points throughout the text, we will consider the action of $[g]$
	on cohomology, in which case it corresponds to the map dual to the
	map $[g]$ on homology, see \autoref{remark:stabilization-description}
	for an explicit description of this map.
\end{notation}

In particular, we use $\stable i {\id} \beta {\id}$ to denote the stable
rational homology of configuration space on $n$ points, which is well known to be
$1$-dimensional if $i = 0$ or $1$ and $0$ dimensional otherwise; 
one may deduce this from the computation of integral homology of
ordered configuration space in \cite{arnold:cohomology-of-braid-group} by rationalizing and taking $S_n$ invariants.

Throughout this paper, it will be crucial to understand the connected components
of our Hurwitz spaces. For our dihedral groups of order $2 \ell$, the 
stable components turn out to be uniquely determined by their boundary
monodromy, as we explain next.

\begin{lemma}
	\label{lemma:components}
	Let $H$ be an odd order abelian group, and $G := H \rtimes \mathbb Z/2
	\mathbb Z$, with the generator of $\mathbb Z/2 \mathbb Z$ acting by
	inversion. We let $c \subset G$ denote the conjugacy class of order $2$
	elements.
There is a map $c^n \to G$ given by sending $(g_1, \ldots, g_n) \mapsto g_1
\cdots g_n$.
This map induces a map $\pi_0(\phurc G n c) \to G$. 
We let $\boundarycphurc G n c g$ denote those connected components of $\cphurc G
n c$ mapping to $g$ under the above map.
For $n$ sufficiently large, there are
$\# \wedge^2 H$ connected components of $\boundarycphurc G n c g$ when $n \bmod
2$ agrees with the image of $g$ in $\mathbb Z/2 \mathbb Z$ under the projection
$G \to G/H \simeq \mathbb Z/2 \mathbb Z$ and $0$ components otherwise.
In particular, if $H \simeq \mathbb Z/\ell \mathbb Z$, for $\ell$ an odd prime,
there is at most one such component.
\end{lemma}
\begin{remark}
	\label{remark:}
	The only case of \autoref{lemma:components} we will use is the case $H = \mathbb Z/\ell \mathbb Z$. 
	It is not too difficult to
	show by hand that when $H = \mathbb Z/\ell \mathbb Z$, there is a unique component of
	$\boundarycphurc G n c g$ when $n \bmod
2$ agrees with the image of $g$ in $\mathbb Z/2 \mathbb Z$ and $0$ components otherwise.
	However, we have opted to prove the statement in the above generality,
	as it appears not to have been completely spelled out in the literature.
\end{remark}

\begin{proof}
	The final statement follows from the first statement because when $H =
	\mathbb Z/\ell \mathbb Z$, $\wedge^2 H$ is the trivial group.

	To prove the first statement, there is a certain finite abelian group $H_2(G,c)$
	defined in \cite[Definition, p.
	3]{wood:an-algebraic-lifting-invariant}, and which we will recall the
	definition of in the
	next paragraph, with a map $S_c \to G$.
	It follows from \cite[Theorem 3.1, Theorem 2.5, and the Definition on p.
	3]{wood:an-algebraic-lifting-invariant} that $H_2(G,c)$
	satisfies the following property:
	the number of irreducible components of $\#\boundarycphurc G n c
	g$ is identified with $\# \ker(S_c \to G)$, when the image of $g$ in the
	abelianization of $G$ agrees with the image of $n$ in the abelianization
	$\mathbb Z/2 \mathbb Z$, and there are no such components otherwise.

	In our present situation, we claim the finite group $H_2(G,c)$ is identified with
	the usual group homology $H_2(H,\mathbb Z)$. This was outlined in
	\cite[9.3.2]{ellenbergVWhomologicalII} and is also closely related to
	the proof of
	\cite[Theorem
	3.1]{sawinW:conjectures-for-distributions-containing-roots-of-unity}. We now recapitulate the
	argument.
	Let $\langle x,y \rangle$ denote the image of the canonical generator
	$[(1,0)|(0,1)] - [(0,1)|(1,0)] \in H_2(\mathbb Z^2, \mathbb Z)$ under
	the map $H_2(\mathbb Z^2, \mathbb Z) \to H_2(G, \mathbb Z)$, induced by
	the map
	$\mathbb Z^2 \to G$ sending the first generator to $x$ and
	the second generator to $y$.
	By definition the group $H_2(G,c)$ is the quotient of
	$H_2(G,\mathbb Z)$ by all classes $\langle x,y \rangle$ for $x \in c$
	such that $x$ and $y$ commute.

	We next show that the quotient map $H_2(G, \mathbb Z) \to H_2(G, c)$ is
	an isomorphism.
	Let $\langle x \rangle$ denote
	the subgroup of $G$ generated by $x \in c$, isomorphic to $\mathbb Z/2 \mathbb
	Z$.
	A direct computation shows that the only elements of $G$ commuting with
	$x$ are $\{x, \id\}$, and in either case, 
	we have $\langle x,y
	\rangle$ lies in the image of $H_2( \langle x \rangle, \mathbb Z) \to
	H_2(G, \mathbb Z)$. For $A$ any abelian group, $H_2(A, \mathbb Z) =
	\wedge^2 A$, so $H_2( \langle x \rangle, \mathbb Z) \simeq \wedge^2
	(\mathbb Z/2 \mathbb Z)$ is the trivial group, and hence $H_2(G, \mathbb Z) \simeq H_2(G,
	c)$.

	It is a standard group cohomology fact that $H_2(H, \mathbb Z) \simeq
	\wedge^2 H$. Hence, to conclude the proof, it suffices to show
	 $H_2(H, \mathbb Z) \simeq H_2(G, \mathbb Z)$.
	Indeed, this follows from the spectral sequence associated to the exact sequence $H
	\to G \to \mathbb Z/2 \mathbb Z$, as we now explain.
	We may observe that $H_0(\mathbb Z/2 \mathbb Z, H_1(H, \mathbb Z))$ is
	the coinvariants of the odd order group $H$ by the inversion action, and
	hence vanishes. Additionally, $H_1(\mathbb Z/2 \mathbb Z, H_1(H, \mathbb
	Z)) = H_2(\mathbb Z/2 \mathbb Z, H_0(H, \mathbb Z)) = 0$ as both groups
	are $2$-torsion but are also $H$-modules, and hence must vanish.
	Therefore, the spectral sequence yields an isomorphism $H_0(\mathbb
	Z/2\mathbb Z, H_2(H, \mathbb Z)	) \simeq H_2(G, \mathbb Z)$.
	Since $\mathbb Z/2 \mathbb Z$ acts trivially on $H_2(H, \mathbb Z)$,
	we find $H_2(H, \mathbb Z) \simeq H_2(G, \mathbb Z)$.
\end{proof}

\begin{definition}
	\label{definition:boundary-monodromy}
	We assume $k$ is a field of characteristic $0$.
	For $G$ a dihedral group of order $2 \ell$ with $\ell$ odd, $g \in G$, and $\phi: G \to \mathbb Z/2 \mathbb Z$ the surjection
	given by quotienting by $\mathbb Z/\ell \mathbb Z$,
	we explained in \autoref{lemma:components} why there is a bijection between components of
	$\cphurc G n c$ and elements $g$ with $\phi(g) \equiv n \bmod 2$ for $n$
	sufficiently large.

We let $\boundarycphurc G n c g$ denote the union of those connected components of $\cphurc G
n c$ mapping to $g$ under the map described in \autoref{lemma:components}.
We say such components have {\em boundary monodromy} $g$.
When $n$ is sufficiently large, it follows from \autoref{lemma:components} that 
$\boundarycphurc G n c g$ is either empty or connected.
We let $\stableboundary i G 
	c g$ denote the subspace of the stable homology $\stable i
	G \beta c$, 
which we can identify with the subspace
$H_i(\boundarycphurc G n c g,k) \subset H_i(\cphurc G n c,k)$
for $n \equiv \beta \bmod 2$ sufficiently large.
If $x \in \stable i G {\phi(g)} c$, we say
	$x$ has {\em boundary monodromy} $g \in G$ if $x$ lies in
	$\stableboundary i G
	c g$.
\end{definition}

\section{Cohomology of the two-sided $\mathcal K$ complex}
\label{section:2-sided}

In this section, we introduce the two-sided $\mathcal K$-complex associated to a
pair of modules and compute its
homology for a particular pair of modules.
In
\autoref{subsection:k-complex},
we first recall the usual ($1$-sided) $\mathcal K$-complex introduced in
\cite{EllenbergVW:cohenLenstra} and refined in
\cite{randal-williams:homology-of-hurwitz-spaces}.
We then define the two-sided $\mathcal K$-complex associated to a pair of
modules in
\autoref{subsection:two-sided-k-complex}.
In \autoref{subsection:easy-nullhomotopy}, we produce a simple nullhomotopy on
certain two-sided $\mathcal K$-complexes.
We then proceed to compute the homology of another two-sided $\mathcal K$-complex
in \autoref{subsection:particular-two-sided}.
The main result from this section we will use in the future is
\autoref{proposition:upper-triangle-exact}.
This is essentially equivalent to
\autoref{theorem:two-sided-g}, which computes the homology of a certain two-sided
$\mathcal K$-complex.
\autoref{theorem:two-sided-g},
which is a bit easier to state than
\autoref{proposition:upper-triangle-exact}, will help
motivate our proof of
\autoref{proposition:upper-triangle-exact}.
The key to both of these results is an explicit nullhomotopy of a large
subcomplex of the two-sided $\mathcal K$-complex given in
\autoref{lemma:nullhomotopy}. 

\subsection{Review of the usual $\mathcal K$-complex}
\label{subsection:k-complex}

Fix a field $k$ of arbitrary characteristic and
let $A := \oplus_{n \geq 0} C_\bullet(\phurc G n c, k)$ denote the algebra of
singular $k$-chains associated to the Hurwitz spaces parameterized by $G$ and
$c$.
We use $R:= \oplus_{n \geq 0} H_0(\phurc G n c, k)$ to denote the ring of
components of Hurwitz spaces.
Note that $A$ has a grading given by the index $n$ parameterizing the number of
branch points.


We say $M$ is a {\em discrete} graded $A$ module if $M$ is a chain complex concentrated
in degree 0, which is a module for $A$.
(In other words, $M$ is just what one usually thinks of as an $R$ module.)
The $\mathcal K$-complex, $\mathcal K(M)$, introduced in 
\cite[4.1]{EllenbergVW:cohenLenstra}
and also 
\cite[p.
16]{randal-williams:homology-of-hurwitz-spaces},
is the
chain complex
\begin{align}
\label{equation:k-complex}
\cdots \to k\{c\}^{\otimes n} \otimes M \to k\{c\}^{\otimes(n-1)} \otimes M \to
\cdots \to k\{c\} \otimes M \to M.
\end{align}
The boundary maps in \eqref{equation:k-complex} are
given by
\begin{equation}
\label{equation:k-differential}
\begin{aligned}
&d(g_1 \otimes \cdots \otimes g_n \otimes [m]) \\
&= \sum_{i=1}^n (-1)^i g_1 \otimes
\cdots g_{i-1} \otimes g_{i+1} \otimes \cdots \otimes g_n \otimes (g_n^{-1}
\cdots g_{i+1}^{-1} g_i g_{i+1} \cdots g_n) \cdot [m].
\end{aligned}
\end{equation}

\begin{remark}
\label{remark:}
Suppose $A \to R :=
\oplus_{n \geq 0} H_0(\phurc G n c,k)$
denotes the quotient map.
If $M$ is an $R$ module,
Ellenberg Venkatesh and Westerland
introduced the $\mathcal K$-complex 
$\mathcal K(M)$ associated to such an $M$ \cite[4.1]{EllenbergVW:cohenLenstra}.
As mentioned in the last paragraph of \cite[p.
16]{randal-williams:homology-of-hurwitz-spaces}, $\mathcal K(M)$ is
the Koszul complex for computing the homology of the derived tensor product $k \otimes_A^{\mathbb L} M$.
\end{remark}

\subsection{Definition the two-sided $\mathcal K$-complex}
\label{subsection:two-sided-k-complex}

Generalizing the $\mathcal K$-complex described in \autoref{subsection:k-complex}, we
next introduce the two-sided $\mathcal K$-complex.
Still letting 
$A := \oplus_{n \geq 0} C_\bullet(\phurc G n c, k)$
as in \autoref{subsection:k-complex}, we now let $M$ be a discrete right $A$
module and $N$ be a discrete left $A$ module.

\begin{definition}
	\label{definition:two-sided-k-complex}
	Define the {\em two-sided $\mathcal K$-complex}, $\mathcal K(M, A, N)$ to be the
double complex with $(i,j)$ term given by
$\oplus_{\alpha + \beta = j} M_\alpha \otimes k\{c\}^i \otimes N_\beta$.
The the total differential on this double complex is the sum of a ``rightward'' differential $d_{\on{r}}$
and a ``leftward'' differential $d_{\on{l}}$ which are given as follows.
The rightward differential is given by 
\begin{equation}
\label{equation:k-right-differential}
\begin{aligned}
&d_{\on{r}}([m] \otimes g_1 \otimes \cdots \otimes g_n \otimes [\omega]) \\
&=
\sum_{i=1}^n (-1)^i [m] \otimes g_1 \otimes
\cdots g_{i-1} \otimes g_{i+1} \otimes \cdots \otimes g_n \otimes (g_n^{-1}
\cdots g_{i+1}^{-1} g_i g_{i+1} \cdots g_n) \cdot [\omega],
\end{aligned}
\end{equation}
while the leftward differential is given by
\begin{equation}
\begin{aligned}
\label{equation:k-left-differential}
&d_{\on{l}}([m] \otimes g_1 \otimes \cdots \otimes g_n \otimes [\omega]) \\
&=
\sum_{i=1}^n (-1)^i [m] \cdot g_i \otimes (g_i^{-1} g_1 g_i) \otimes
\cdots (g_i^{-1} g_{i-1} g_i) \otimes g_{i+1} \otimes \cdots \otimes g_n \otimes
[\omega].
\end{aligned}
\end{equation}
\end{definition}

It may be helpful to refer to
\autoref{figure:two-sided} for a visualization of a particular summand of a two-sided
$\mathcal K$ complex.

The next remarks are not needed in what follows, but may serve as some motivation
for our above definition.
\begin{remark}
	\label{remark:}
	There is an alternate, more abstract definition of the two sided
	$\mathcal K$ complex, $\mathcal K(M, A, N)$. Namely, it is equivalent to
	the derived tensor product $M \otimes_A^{\mathbb L} N$. To see this,
	note that $M \otimes_A^{\mathbb L} N$ has a double
	filtration, induced by filtering both $M$ and $N$ by
their gradings.
Taking the associated graded with respect to both filtrations, we get $(k
\otimes_A^{\mathbb L} k) \otimes_k (M \otimes_k N)$. Similarly to
\autoref{subsection:k-complex}, running the spectral sequence in each filtration
direction, we find that the spectral sequences collapses at the $E_2$-page
because of the discreteness of $M$ and $N$. 
We thus obtain a double complex whose underlying bigraded vector space has
homology agreeing with
$(k \otimes_A^{\mathbb L} k) \otimes_k (M \otimes_k N)$ and
whose total complex is quasi-isomorphic to $M \otimes_A^{\mathbb L} N$.
Concretely, one can show that the $E^1$ page of this spectral sequence
can be viewed as a double complex whose $(i,j)$th term 
agrees with $\mathcal K(M, A, N)_{i,j}$.
Moreover, the differential $\mathcal K(M,N)_{i,j}$ to
$\mathcal K(M,N)_{i-1,j}$
can be shown to agree with the differential on this $E^1$ page
of this spectral sequence
(see
\cite[Theorem 6.2 and p. 16]{randal-williams:homology-of-hurwitz-spaces}).
\end{remark}

\begin{remark}
	In the case that $M$ and $N$ are modules coming from the action of the
	Hurwitz space on a set, the two-sided $\mathcal{K}$-complex can be
	viewed as a cellular chain complex for the spaces in \cite[Theorem
	A.4.9]{landesmanL:the-stable-homology-of-non-splitting}.
\end{remark}

\subsection{A nullhomotopy for two-sided $\mathcal K$ complexes}
\label{subsection:easy-nullhomotopy}

We next show that certain types of two-sided
$\mathcal K$ complexes are exact.
The proof of the following lemma is inspired by
\cite[Lemma 4.11]{EllenbergVW:cohenLenstra}.
One can also deduce this by using \cite[Lemma
4.3.1]{landesmanL:the-stable-homology-of-non-splitting}.

\begin{lemma}
	\label{lemma:zero-and-iso-vanishes}
Let $A := \oplus_{n \geq 0} C_\bullet(\phurc G n c, k)$
	Suppose $M$ is a left $A$ module and $N$ is a right $A$ module.
	Assume there is some $h \in c$ so that either
	\begin{enumerate}
		\item $[h]$ acts invertibly on $M$ and by $0$ on $N$, or
		\item $[h]$ acts invertibly on $N$ and by $0$ on $M$.
	\end{enumerate}
	Then, there is a chain homotopy between the identity map on $\mathcal
	K(M, A,N)$ and the $0$ map on $\mathcal K(M, A, N)$. In particular, $\mathcal
	K(M, A,N)$ is exact.
\end{lemma}
%
\begin{proof}
	We just explain the second case that $[h]$ acts invertibly on $N$ and by
	$0$ on $M$, as the first case is analogous.
	We now give an explicit chain homotopy between the identity and $0$.

	We start with a warm up computation. Define $S^0_{n,j}$ by 
	\begin{align*}
		S^0_{n,j}: \mathcal K(M,A,N) & \rightarrow \mathcal K(M,A,N) \\
		[m] \otimes g_1 \otimes \cdots \otimes g_n \otimes
	[\omega] &\mapsto (-1)^{n+1} [m] \otimes g_1 \otimes \cdots
		\otimes g_n \otimes h \otimes
		[h]^{-1} \cdot [\omega].
	\end{align*}
	Note that the definition of this chain map uses that $[h]$ acts
	invertibly on $N$.
	Using that $[g_1][g_1^{-1} g_2 g_1] = [g_2][g_1]$, a routine computation
	similar (but easier than) the computation in
	\autoref{lemma:nullhomotopy}
	verifies that
	\begin{align*}
		&(d S^0_{n,j} + S^0_{n-1,j} d)([m] \otimes g_1 \otimes \cdots \otimes g_n \otimes
		[\omega]) 
		\\ &= [m] \otimes g_1 \otimes \cdots \otimes g_n \otimes
		[\omega] + [m][h] \otimes h^{-1} g_1 h \otimes \cdots
		\otimes h^{-1} g_n h \otimes
		[h]^{-1} [\omega].
	\end{align*}
	Now, using that 
	$[m] \cdot [h] = 0$,
	the map $S_{n,j}^0$ gives a chain homotopy between the identity and
	the $0$ map. This implies that
	$\mathcal K(M,A,N)$ is exact.
\end{proof}

\begin{notation}\label{notation:module}
	For $g \in c$, let $k[g]$ be the graded $A$-bimodule that is $k$ in each
	natural number degree, such that left and right multiplication by $[g]$
	acts by shifting. Moreover, for any $g\neq h \in c$, $[h]$ acts by $0$.
	We use $k[g,g^{-1}]$ to denote the analogous graded $A$ modules with integer
	gradings. (So right and left multiplication by $[g]$ shifts the grading
	while $[h]$ acts by $0$ for $h \neq g$.)
\end{notation}
Note there are natural graded bimodule maps $R \to k[g] \to k[g,g^{-1}]$, where the first map is given by quotienting out by elements that are not multiples of $g$. 
\begin{lemma}
\label{lemma:g-stable-k-complex}
The complex $\mathcal K(k, A, k[g, g^{-1}])$ is exact.
In particular, the complex $\mathcal K(M)$ associated to $M := k[g,g^{-1}]$, as defined in
\eqref{equation:k-complex}, is exact.
\end{lemma}
\begin{proof}
	Since $g$ acts by $0$ on $k$ and invertibly on $k[g,g^{-1}]$, it follows
	that $\mathcal K(k,A,k[g,g^{-1}])$ is exact.
	Since $\mathcal K(k,A,k[g,g^{-1}])$ is a sum of $\mathbb Z$ many shifted
	copies of $\mathcal K(M)$, so that, in particular, $\mathcal K(M)$ is a
	summand of $\mathcal K(k,A, k[g,g^{-1}])$,
	we obtain $\mathcal K(M)$ is exact.
\end{proof}

\subsection{A particular two-sided $\mathcal K$-complex}
\label{subsection:particular-two-sided}

In this subsection, we investigate the homology of a particular two-sided
$\mathcal K$ complex.
This two-sided $\mathcal K$ complex will appear in the homology of Hurwitz
spaces, and understanding its homology is the crucial step in understanding the
stable homology of Hurwitz spaces.
\begin{notation}
\label{notation:k-complex}
We fix a group $G$ and a conjugacy class $c \subset G$ generating
$G$ and a field $k$.
We assume that the for any $h \in c$, the centralizer of $h$ in $c$ is
precisely $h$. 
We let $A := \oplus_{n \geq 0} C_\bullet(\phurc G n c, k)$ denote the algebra of
singular $k$-chains associated to the Hurwitz spaces parameterized by $G$ and
$c$.
We let $B := C_*(\conf_n, k)$ denote the corresponding algebra of
singular $k$-chains on configuration space.
\end{notation}
\begin{remark}
\label{remark:}
We note that this condition on the centralizer in \autoref{notation:k-complex} 
will hold for dihedral groups, with $c$ the conjugacy class of involutions.
\end{remark}

Our goal in this section is to compute the homology of $\mathcal K( k[g],A, k[g,g^{-1}])$.
With notation as in \autoref{notation:k-complex},
There is a map $B \to A$ induced by the inclusion $\conf_n \to
\cphurc G n c$ with image the component of Hurwitz space parameterizing covers
whose monodromy is $g$ at every branch point. (These will be
disconnected covers, unless $G$ is already a cyclic group.)
\begin{theorem}
\label{theorem:two-sided-g}
The map $\mathcal K(k[g], B, k[g, g^{-1}]) \to \mathcal K(k[g], A, k[g, g^{-1}])$, induced by the above map $B \to A$, is an isomorphism.
\end{theorem}
We will prove \autoref{theorem:two-sided-g} in
\autoref{subsection:proof-two-sided-theorem}.
The following is a direct consequence of the spectral sequence described in
\autoref{subsection:two-sided-k-complex}.

\begin{figure}
	\begin{tikzcd}[column sep=small, row sep=small]
		\qquad &k\{c^0\} && k\{c^0\}&&
		k\{c^0\} &&
		k\{c^0\} && \\
		\cdots \ar {ur} && k\{c^1\} \ar{ur}\ar{ul} && k\{c^1\}
		\ar{ur}\ar{ul} && k\{c^1\} \ar{ur}\ar{ul} && \\
		\qquad & k\{c^2\} \ar{ur}\ar{ul} && k\{c^2\}\ar{ur}\ar{ul}&&
		k\{c^2\}\ar{ur}\ar{ul} && \\
		\cdots \ar {ur} && k\{c^3\} \ar{ur}\ar{ul} && k\{c^3\}
		\ar{ur}\ar{ul} && & \\
		\qquad & k\{c^4\} \ar{ur}\ar{ul} && k\{c^4\} \ar{ur}\ar{ul} &&
		&&  \\
		\cdots \ar{ur} && k\{c^5\} \ar{ur}\ar{ul} &&
		&& &&  \\
		\qquad & k\{c^6\} \ar{ur} \ar{ul} &&  && &&  \\
		\cdots \ar{ur} &&   && && &&  \\
	\end{tikzcd}
	\caption{The complex above depicts the summand $C_{\bullet,\bullet}$ (defined in
	\autoref{remark:z-summand}) of the
two sided $\mathcal{K}$-complex, $\mathcal K(k[g], A,k[g,g^{-1}])$.}
	\label{figure:two-sided}
\end{figure}
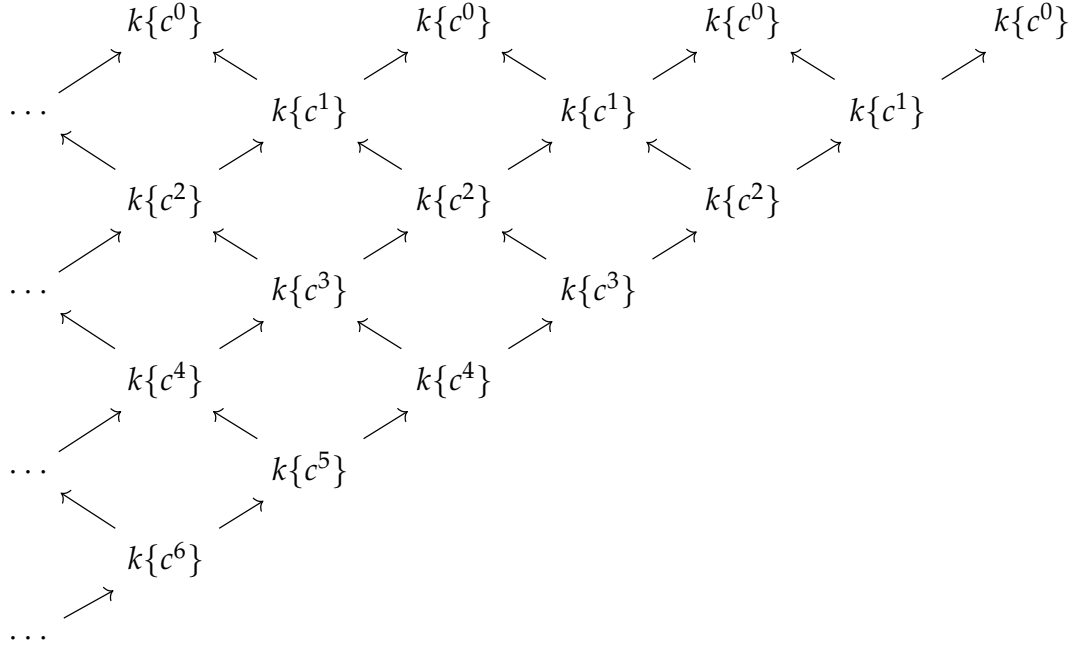

\begin{lemma}
	\label{lemma:two-sided-k-complex-descrption}
	There is a collection of double complexes $C^z_{\bullet, \bullet}$,
	indexed by $z \in \mathbb Z$ so that
	the sum of the associated total complex,
	$\bigoplus_{z \in \mathbb Z} C^{z}_{\bullet, \bullet}$,
	viewed as an object in the derived
	category, is isomorphic to
	$\mathcal K(k[g], A,k[g,g^{-1}])$.
	Specifically, $C^z_{i, j} = k\{c\}^{z - i -j}$ 
	and there are differentials
	$C^z_{i,j} \to C^z_{i, j+1}$ and $C^z_{i,j} \to C^z_{i+1,j}$ given by the maps
$d_{\on{r}}, d_{\on{l}}$ of \eqref{equation:k-right-differential}
and \eqref{equation:k-left-differential} associated to the modules $M = k[g], N=
k[g,g^{-1}]$ from \autoref{notation:module}.
\end{lemma}

\begin{remark}
	\label{remark:z-summand}
	It follows immediately from the description of
	\autoref{lemma:two-sided-k-complex-descrption} that $C^z_{s,u}$ can be identified with those
terms of the form $g^s \otimes k\{c\}^{t} \otimes g^u$ where $s + t + u = z$ and
$s \geq 0$.
Additionally, $C^z_{i,j} \simeq C^{z'}_{i,j+z-z'},$ compatibly with the
differentials, 
for any
$z, z'$, as follows from the definition. 
Hence, it follows that these two complexes are isomorphic, up to a shift.
To simplify notation a bit, we write  $C_{\bullet, \bullet} :=C^0_{\bullet,
\bullet}$. See \autoref{figure:two-sided} for a picture of this complex.
\end{remark}

\begin{remark}
\label{remark:concrete-g-differentials}
Let $\delta_{h,h'}$ be $1$ if $h = h'$ and $0$ otherwise,
Using the definition of the differentials for the two-sided $\mathcal K$ complex
given in \autoref{definition:two-sided-k-complex},
we can describe the differentials on $C_{\bullet, \bullet}$ as
\begin{equation}
\label{equation:g-k-right-differential}
\begin{aligned}
&d_{\on{r}}([g^\alpha] \otimes g_1 \otimes \cdots \otimes g_n \otimes [g^\beta]) \\
&=
\sum_{i=1}^n \delta_{g, g_n^{-1}
\cdots g_{i+1}^{-1} g_i g_{i+1} \cdots g_n} (-1)^i [g^\alpha] \otimes g_1 \otimes
\cdots g_{i-1} \otimes g_{i+1} \otimes \cdots \otimes g_n \otimes [g^{\beta+1}].
\end{aligned}
\end{equation}
and
\begin{equation}
\label{equation:g-k-left-differential}
\begin{aligned}
&d_{\on{l}}([g^\alpha] \otimes g_1 \otimes \cdots \otimes g_n \otimes [g^\beta]) \\
&=
\sum_{i=1}^n \delta_{g, g_i}(-1)^i  [g^{\alpha+1}] \otimes (g_i^{-1} g_1 g_i) \otimes
\cdots (g_i^{-1} g_{i-1} g_i) \otimes g_{i+1} \otimes \cdots \otimes g_n \otimes
[g^\beta].
\end{aligned}
\end{equation}
\end{remark}

\begin{notation}
	\label{notation:t-definition}
	Let $C'_{\bullet, \bullet}$ denote the subcomplex of $C_{\bullet, \bullet}$
which is spanned by all basis elements $(g_1, \ldots, g_n) \in
c^n$ with some $g_i \neq g$.
In particular, $\dim C'_{i,j} = |c ^{i+j}| - 1$.
Let $T_n$ denote the total complex $\oplus_{i+j = n} C'_{i,j}$
associated to $C'_{i,j}$.
\end{notation}

Consider a basis element $v$ in $C'_{i,j}$ of the form
\begin{align}
	\label{equation:g-form}
	(x_1, \ldots, x_s, h, \underbrace{g, \ldots, g}_{t \text{ times}})
\end{align}
with $h \neq g$.
Define a filtration $F^i_n \subset T_n$ spanned by those $v$ as above with $t
\leq i$.
Consider the map $\sigma_n : T_n \to T_{n+1}$ sending
\begin{align*}
	(x_1, \ldots, x_s, h, \underbrace{g, \ldots, g}_{t \text{ times}})
	\mapsto 
	(-1)^n \cdot (x_1, \ldots, x_s, hgh^{-1}, h, \underbrace{g, \ldots, g}_{t \text{
	times}}).
\end{align*}

\begin{remark}
	\label{remark:alternate-homotopy}
	In \cite{landesmanL:the-stable-homology-of-non-splitting} we use a geometric
	map which corresponds algebraically to the chain homotopy
\begin{align*}
	(x_1, \ldots, x_s, h, \underbrace{g, \ldots, g}_{t \text{ times}})
	\mapsto 
	(-1)^n \cdot (gx_1g^{-1}, \ldots, gx_s g^{-1}, g, h, \underbrace{g, \ldots, g}_{t \text{
	times}}).
\end{align*}
Similarly to \autoref{lemma:nullhomotopy}, one can also verify this defines a chain homotopy between an isomorphism
and $0$. This alternate chain homotopy was pointed out to us by
Andrea Bianchi.
\end{remark}

\begin{lemma}
	\label{lemma:nullhomotopy}
	The map
$\sigma_\bullet: T_\bullet \to T_{\bullet + 1}$ defines a
chain homotopy between an isomorphism and $0$.
Hence, the complex $T_\bullet^n$ is exact.
\end{lemma}
\begin{proof}
	If $\sigma_\bullet$ defines a chain homotopy between an
	isomorphism and $0$, this implies $0$ acts the same on cohomology as an
	isomorphism, and so $T_\bullet^n$ is exact.

	We now show $\sigma_\bullet$ defines a chain homotopy between an
	isomorphism and $0$.
	Explicitly we wish to show $\sigma_{n-1} d_n + d_{n+1} \sigma_n$ induces
	an isomorphism $T_n \to T_n$.
	To prove this, we will first show that this map preserves the filtration
	$F^i_n$, and then that it induces isomorphisms on the associated graded
	of the filtration $F^i_n/F^{i-1}_n$.

	First, let us verify the map $\sigma_{n-1} d_n + d_{n+1} \sigma_n$
	preserves the filtration $F^i_n$.
	Note that the map $\sigma_{j}$ sends $F^i_j$ to $F^i_{j+1}$ by
	construction. Therefore, to show $\sigma_{n-1} d_n + d_{n+1} \sigma_n$
	preserves the filtration, it is enough to show $d_j$ sends $F^i_j$ to
	$F^i_{j-1}$.
	Moreover, 
	since the differential $d$ is a sum of $d_{\on{l}} + d_{\on{r}}$, it is
	enough to show that $d_{\on{l}}$ and $d_{\on{r}}$ separately preserve
	the filtration.
	Using the definition of these differentials appearing in 
	\eqref{equation:g-k-right-differential} and
	\eqref{equation:g-k-left-differential},
	we see each such differential is a sum of terms associated to each of
	the $n$ entries of $v$, so it is enough to show each of these terms lies
	in $F^i_{j-1}$. 
	First, let us analyze the terms in $d_{\on{l}}$. For the terms
	associated to $x_1, \ldots, x_s$ for $v$ as in \eqref{equation:g-form},
	these terms do not alter the $t$ entries $g$ on the right, and so the
	filtration is preserved for these terms.
	For the term corresponding to $h$, (in position $s + 1$,) since $h \neq g$, this term must be $0$.
	Finally, each of the terms associated to one of the rightmost $t$
	entries $g$ lie in the filtration $F^{i-1}_{j-1}$, which indeed lies in
	$F^i_{j-1}$.
	To conclude, we show $d_{\on{r}}$ sends $F^i_j$ to $F^i_{j-1}$.
	For each of the terms associated one of the $t$ rightmost $g$'s, this
	sends $F^i_j$ to $F^{i-1}_{j-1}$, which indeed lies in $F^i_{j-1}$.
	The term associated to $h$ vanishes because the conjugate of $h$ by any
	power of $g$ is not $g$ since $h \neq g$.
	Finally, we note that the term associated to one of the $x_1, \ldots, x_s$, for $x_i$
	as in \eqref{equation:g-form} lies in $F^i_{j-1}$ because the term is
	either $0$ or of the form $\pm (x_1, \ldots, x_{i-1}, x_{i+1}, \ldots,
	x_s, h, g, \ldots, g)$.

	Having shown the filtration is preserved, we conclude by checking that 
	the map induced by $\sigma_{n-1} d_n + d_{n+1} \sigma_n$ on
	$F^t_n/F^{t-1}_n$ is a multiple of the identity.
	In other words, we assume there are precisely $t$ elements $g$ to the
	right of $h$, for $h$ the rightmost term not equal to $g$.
	Via the analysis above, the terms associated to the right most $t+1$
	terms in $d_{\on{l}}$ and $d_{\on{r}}$ both vanish, and hence we only
	need analyze the remaining terms.
	First, we will compute the map induced by $\sigma_{n-1} d_n=\sigma_{n-1}
	d_{\on{r}} + \sigma_{n-1} d_{\on{l}}$.
	Using the definition of the differential from
	\eqref{equation:g-k-left-differential},
	\begin{align}
		\label{equation:dl-then-sigma}
		\sigma_{n-1} d_{\on{l}} (v) =
		(-1)^{n-1} \sum_{i=1}^s \delta_{g, x_i} (-1)^i \left(x_i^{-1} x_1 x_i, \ldots, x_i^{-1}
			x_{i-1} x_i, x_{i+1}, \ldots, x_s, hgh^{-1}, h,
	\underbrace{g, \ldots, g}_{t \text{ times}} \right).
	\end{align}
	Similarly, using the definition of the right differential from
	\eqref{equation:g-k-right-differential},
	\begin{align}
		\label{equation:dr-then-sigma}
		\sigma_{n-1} d_{\on{r}}(v) = 
		(-1)^{n-1} \sum_{i=1}^s &\delta_{g, (x_{i+1} \cdots x_s h g^t)^{-1} x_i
		(x_{i+1} \cdots x_s h g^t)} (-1)^i \\
		&\left(x_1, \ldots, x_{i-1}, x_{i+1}, \ldots,
			x_s, hgh^{-1}, h, \underbrace{g, \ldots, g}_{t \text{ times}} \right).
	\end{align}

	Next, we compute the map induced by 
	$d_{n+1} \sigma_{n}=
	d_{\on{r}} \sigma_{n}+ d_{\on{l}}\sigma_{n}$.
	Again, we have two computations, which are nearly the same as
	\eqref{equation:dl-then-sigma} and \eqref{equation:dr-then-sigma}.
	First, 
	\begin{align*}
\sigma_{n}(v) = \left(x_1, \ldots, x_s, hgh^{-1}, h,\underbrace{g, \ldots, g}_{t \text{ times}} \right).
	\end{align*}
  	Therefore,
	\begin{align}
		\label{equation:sigma-then-dl}
		d_{\on{l}} \sigma_n(v) =
	(-1)^{n} \sum_{i=1}^s \delta_{g, x_i} (-1)^i \left(x_i^{-1} x_1 x_i, \ldots, x_i^{-1}
			x_{i-1} x_i, x_{i+1}, \ldots, x_s, hgh^{-1}, h,
	\underbrace{g, \ldots, g}_{t \text{ times}} \right).
	\end{align}
	Note here that a priori there could have been an additional term
	associated to $hgh^{-1}$, but this vanishes because
	$\delta_{g,hgh^{-1}} = 0$ since
	$hgh^{-1}\neq g$, using the assumption from \eqref{notation:k-complex}
	that $g$ is its own centralizer and $h \neq g$.
	We see from \eqref{equation:dl-then-sigma} and
	\eqref{equation:sigma-then-dl} that 
	\begin{align}
		\label{equation:left-cancel}
	d_{\on{l}} \sigma_n +  \sigma_{n-1} d_{\on{l}} = 0,
	\end{align}
	since they have opposite signs.

	Finally, we compute
	\begin{equation}
		\label{equation:sigma-then-dr}
		\begin{aligned}
		d_{\on{r}}\sigma_{n} (v) = 
		(-1)^n \sum_{i=1}^s &\delta_{g, (x_{i+1} \cdots x_s (h gh^{-1}) h  g^t)^{-1} x_i
		(x_{i+1} \cdots x_s (h gh^{-1}) h  g^t)} (-1)^i \\
		&\left(x_1, \ldots, x_{i-1}, x_{i+1}, \ldots,
			x_s, hgh^{-1}, h, \underbrace{g, \ldots, g}_{t \text{ times}} \right)
			\\
			&+ (-1)^n (-1)^{s+1} (x_1, \ldots, x_s, h, \underbrace{g,
			\ldots, g}_{t \text{ times}}).
		\end{aligned}
	\end{equation}

	Now, we claim the terms in the sum over $i$ in
	\eqref{equation:sigma-then-dr} precisely cancels with the corresponding
	terms in \eqref{equation:dr-then-sigma}.
	To see this, we need only observe that 
	\begin{align*}
		\delta_{g, (x_{i+1} \cdots x_s (h gh^{-1}) h  g^t)^{-1} x_i
		(x_{i+1} \cdots x_s (h gh^{-1}) h  g^t)}
		=\delta_{g, (x_{i+1} \cdots x_s h g^t)^{-1} x_i
		(x_{i+1} \cdots x_s h g^t)}
	\end{align*}
	In other words, we wish to show 
	\begin{align*}
	(x_{i+1} \cdots x_s (h gh^{-1}) h  g^t)^{-1} x_i
		(x_{i+1} \cdots x_s (h gh^{-1}) h  g^t) = g
	\end{align*}
 if and only if
 \begin{align*}
	(x_{i+1} \cdots x_s h g^t)^{-1} x_i
(x_{i+1} \cdots x_s h g^t) = g.  
 \end{align*}
Conjugating both sides by $g^{t+1}$ in the
first equation, we obtain it is equivalent to 
\begin{align*}
	(x_{i+1} \cdots x_s h)^{-1} x_i
(x_{i+1} \cdots x_s h) = g.
\end{align*}
This is also equivalent to the second equation conjugated by
$g^t$. Hence, the two are equivalent, as claimed. 
This implies 
\begin{align}
	\label{equation:right-cancel}
	(d_{\on{r}} \sigma_n +  \sigma_{n-1} d_{\on{r}}) (v) = 
(-1)^n (-1)^{s+1} (x_1, \ldots, x_s, h, \underbrace{g,
			\ldots, g}_{t \text{ times}})
			= (-1)^{n+s+1} v.
\end{align}

Therefore, all in all, adding \eqref{equation:left-cancel} and
\eqref{equation:right-cancel}, we find that the map induced by 
$\sigma_{n-1} d_n + d_{n+1} \sigma_n$ is given by $v \mapsto (-1)^{n+s+1}(v)$, and
therefore is an isomorphism, as we wished to show.
\end{proof}

The key result we will need from this section is the following consequence of
the fact that $T^n_\bullet$ is exact.
\begin{proposition}
	\label{proposition:upper-triangle-exact}
	For $x \in \mathbb Z_{\leq 0}$, let $D^{x}_{\bullet, \bullet}$ denote the subcomplex of $C'_{\bullet,
	\bullet}$ (defined in \autoref{notation:t-definition}) 
	which is equal to $C'_{i,j}$ if $i \geq 0$ and $j \geq x$ and
	is $0$ otherwise.
	(Pictorially, $D^{x}_{\bullet, \bullet}$ makes
		a cone shape above $C'_{0,x} \simeq k\{c^{-x}\}$ in the picture
	\autoref{figure:two-sided}.)
	Then, $D^{x}_{\bullet, \bullet}$ is exact except possibly at position $(0,x)$.
\end{proposition}
\begin{proof}
	By \autoref{lemma:nullhomotopy}, $T_\bullet^n$ is exact, or equivalently
	$C'_{\bullet, \bullet}$ is exact.
	Define the complex $J_\bullet$ as follows: Let $J_{-i} := \ker(C_{0,-i}
	\to C_{1,-i})$.
	Define the differential $J_{-i} \to J_{-i+1}$ to be given by $d_{\on{r}}$ by
	viewing $J_{-i} \subset C_{0,-i}$.

	We claim $J_\bullet \simeq C'_{\bullet, \bullet}$. We use that $C'_{\bullet, \bullet}$ has a filtration whose $i$th term is
	$D^{-i}_{\bullet, \bullet}$. The associated graded of this filtration is
	given by the complex $C'_{\bullet, -i}$. (This corresponds to a sequence
		of vector spaces pointing diagonally up and left in
	\autoref{figure:two-sided}, starting with $k\{c^i\}$.)
	Since $C'_{\bullet, -i}$ is a truncated summand of the usual $\mathcal K$
	complex, it is exact except at $C'_{0,-i}$ by
	\autoref{lemma:g-stable-k-complex}. Moreover, the homology of
	$C'_{\bullet, -i}$ is
	precisely $J_{-i}$.
	This implies $J_\bullet \to C'_{\bullet, \bullet}$ is a
	quasi-isomorphism, since it is a map inducing a quasi-isomorphism
	on each associated graded part of the filtration $D^{-i}_{\bullet,
	\bullet}$ of $C'_{\bullet, \bullet}$.
	The same argument moreover shows that the subcomplex $J_{\geq x}$,
	consisting of $J_t$ for $t \geq x$ and $0$ for $t < x$,
	is quasi-isomorphic to the subcomplex
	$D^x_{\bullet, \bullet}$.

	Note that, $J_\bullet$ is exact because it is quasi-isomorphic to
	$C'_{\bullet, \bullet}$, which is an
	exact complex, by \autoref{lemma:nullhomotopy}.
	It follows that $J_{\geq x}$ has only a single nonzero homology group, which occurs in degree $x$. Hence, the
	same is true of $D^x_{\bullet, \bullet}$. Concretely, this means that
	$D^x_{\bullet, \bullet}$ is exact except possibly at position $(0,x)$.
\end{proof}

Combining what we have done so far, we deduce
\autoref{theorem:two-sided-g}.

\subsubsection{Proof of \autoref{theorem:two-sided-g}}
\label{subsection:proof-two-sided-theorem}

\begin{proof}
	
	Observe that the map $\Conf \to \Hur^{G,c}$ in \autoref{theorem:two-sided-g} is induced by functoriality of the Hurwitz space construction for $\Hur^{\langle g\rangle ,\{g\}} \to \Hur^{G,c}$.
	Applying this functoriality with \autoref{lemma:two-sided-k-complex-descrption} and
	\autoref{remark:z-summand}, we can identify $\mathcal K(k[g], B, k[g,
	g^{-1}])
	\to \mathcal K(k[g], A, k[g,g^{-1}])$ as quasi-isomorphic to the inclusion of a summand of 
	$\oplus_{z} C^z_{\bullet, \bullet}$, where the cokernel of this inclusion is
	isomorphic to a sum over $z$ of complexes isomorphic to shifts
	of $T_\bullet$.
	The result then follows from exactness of $T_\bullet^n$, proven in \autoref{lemma:nullhomotopy}. 
\end{proof}

\section{Computing the cohomology stabilized by a single monodromy}
\label{section:stabilized-by-one-element}

In this section, we compute the stable homology of Hurwitz space stabilized by a
single element $g \in c$, and killed by all other elements of $c$.
The main result of this section is \autoref{proposition:cocycle-form}, which has
a somewhat elaborate proof which we break into steps.
The aim of this proposition is to show we can show that a cocycle well into the stable range can be put into a
particularly simple form, and we accomplish by using our computation of the
homology of the two-sided $\mathcal K$-complex from the previous section.
We recognize this argument is a bit involved, so we run through the computation
for the stable first cohomology
in \autoref{example:h1-argument}.
It may be helpful to read this before going through the more general argument.

We fix $G, c, k$ as in \autoref{notation:k-complex}.
We now introduce the
Fox-Neuwirth/Fuks cell complex which is a cell complex computing the cohomology
of Hurwitz space.
Define $\fncell G {n} c i$ to be the free $k$-vector space spanned by tuples of
$n-i$ words in $c$ whose total length is $n$.
That is, a basis element of $\fncell G {n} c i$ is of the form 
\begin{align}
	\label{equation:w-form}
	w_1\otimes
\cdots \otimes
w_{n-i} 
\end{align}
where $w_j$ is a word of length $v_j$ such that $\sum_{j=1}^{n-i}
v_j = n$.
This chain complex has a differential
\begin{align*}
\delta(w_1 \otimes \cdots \otimes w_{n-i}) = \sum_{j=1}^{n-i-1}(-1)^{j+1}
(w_1 \otimes \cdots \otimes  w_{j-1}\otimes  \on{sh} (w_j,  w_{j+1})\otimes
w_{j+2} \otimes \cdots \otimes w_{n-i})
\end{align*}
where $\on{sh} (w_j,  w_{j+1})$ is the shuffle product defined explicitly as
follows:
Suppose $w = g_1 \cdots g_s$ and $ w' = h_1 \cdots h_t$.
Then 
$\on{sh}(w,w')$ is the sum of words of length $s + t$
\begin{align*}
\on{sh}(w,w') = \sum_{\sigma \in S_{s,t}} \sgn(\sigma) \on{sh}_\sigma(g_1, \ldots, g_s,
h_1, \ldots, h_t),
\end{align*}
with the notation above defined as follows:
We let $S_{s,t} \subset S_{s+t}$ denote the subset of permutations of $s + t$ elements
which preserves the relative order of
the first $s$ elements and preserves the relative order of the last $t$
elements. We let $\sgn(\sigma)$ denote the sign of $\sigma$ when viewed as an element
of $S_{s+t}$. 
Finally,
$\on{sh}_\sigma(g_1, \ldots, g_s,
h_1, \ldots, h_t),$ denotes the word of length $s+t$ whose $\sigma(i)$th letter is $h_{i-s}$ for $i
>s$ and is $\alpha_i^{-1} g_{i} \alpha_i$ for $i < s$, where $\alpha_i = h_1
\cdots h_v$, for $v$ is the largest positive integer satisfying $\sigma(s+v) <
\sigma(i)$.
By \cite[Theorem 3.3]{ellenbergTW:fox-neuwirth-fuks-cells}, we can identify the
cohomology of this cochain complex with the cohomology of the Hurwitz space
$\phurc G n c$,
using the isomorphism between the homology of the $1$-point
compactification of a space and the cohomology of that space.
We note that \cite[Theorem 3.3]{ellenbergTW:fox-neuwirth-fuks-cells} is stated
in the language of local systems on configuration space, and if $f: \phurc G n
c \to \conf_n$ denotes the finite covering space, we use
the isomorphism
$H^i(\phurc G n c, \mathbb Q) \simeq H^i(\conf_n, f_*
\mathbb Q)$.

\begin{remark}
	\label{remark:stabilization-description}
	Suppose we start with an element 
	$z \in \fncell G n c i$
	which we write as $\sum_{j=1}^f c_j w_1^j \otimes \cdots \otimes w_{n-i}^j$.
	Recall that we defined a map $[g]$ on homology in
	\autoref{notation:stable-homology}.
	It follows from the definition above that multiplication by the dual of
	the element $[g]$ corresponds to sending
	$z$ to $\sum_{j | w_{n-i}^j = g } c_j w_1^j \otimes \cdots \otimes
	w_{n-i-1}^j$.
	By abuse of notation, we also denote the map dual to $[g]$ by $[g]$, so
	that
	this dual $[g]$ is a map on cohomology.
	In other words, the costabilization by $[g]$ map on cohomology picks out all terms whose last
	word is the length $1$ word equal to $g$.
	Indeed, this is dual to the map on homology which sends a tensor of
	words (in the dual basis to that described above)
	to that same tensor with an additional $g$ tacked on at the end.
\end{remark}

The next proposition is the main result of this section.
It will enable us to run the inductive step and is really the crux of our argument.
The proof will be given later in
\autoref{subsubsection:proof-of-stabilize-by-one}.

\begin{proposition}
	\label{proposition:stabilize-by-one-element-trivial}
	With notation for the stable homology groups as in
	\autoref{definition:boundary-monodromy} and
	\autoref{notation:stable-homology},
	let $i \geq 0$ and
	consider an element $z \in \ker(\stableboundary i G c g
	\to\stable i {\id} \beta {\id})$ (where $\id$ denotes the trivial group/element).
Assume that
for any $j < i$ and any $g \in G$,
	$\stableboundary j G
	c g \to \stable j {\id} \beta {\id}$ is an isomorphism.
Assume that $z [g]^w[h] = 0$ for any $w \geq 0$ and $h \neq g$.
	Then $z = 0$.
\end{proposition}

Before reading the rest of this section, we recommend the reader jump to
\autoref{example:h1-argument}, which runs through the special case of
\autoref{proposition:stabilize-by-one-element-trivial} where $i = 1$.

The next two lemmas show that in order to prove
\autoref{proposition:stabilize-by-one-element-trivial}, it is enough to show that $z$ is cohomologous to a cocycle in a form
ending in $g$; the argument for this reduction is given in
\autoref{subsubsection:proof-of-stabilize-by-one}.

\begin{lemma}
\label{lemma:one-more-g}
Suppose $x \in H^i(\cphurc G {n+1} c, k)$ is represented by a cocycle of the form
$y \otimes g$ for some $y \in \fncell G n c i$.
Then, $y$ is a cocycle and there is
$z \in \fncell G {n-1} c i$
and
$w \in \fncell G {n-1} c {i-1}$
so that
$y - z \otimes g = \delta(w \otimes g)$.
In particular, $y$ is cohomologous to $z \otimes g$ and additionally $z \in
\fncocycle G {n-1} c i$.
\end{lemma}
\begin{proof}
	Note that the final sentence follows immediately from the penultimate sentence.

	To prove the penultimate sentence,
we can write $y = \sum_{j=0}^i \sum_{\tau} s_j^\tau \otimes t_j^\tau$ for
$s_j^\tau \in \fncell G {n-j-1} c {i - j}$ ranging over a basis of this vector
space (as $\tau$ varies for $j$ fixed)
and
$t_j^\tau \in \fncell G {j+1} c {j}$ is a linear combination of words in $c$ of length $j$.

By assumption, $y \otimes g$ is a cocycle, and hence vanishes under the
shuffling coboundary
map $\delta$.
This immediately implies $y$ is a cocycle, since none of the terms associated to
shuffling together blocks in $y$ can cancel with those shuffling the rightmost
block in $y$ with $g$.
In particular, we claim this implies $t_j^\tau \otimes g \in \fncell
G {j+2} c j$ lies in the kernel of the coboundary map $\delta$
for each $j$ with $0 \leq j \leq i$ and each $\tau$. 
Indeed, this can be seen by expanding the image of $y \otimes g$ under
the coboundary map and noting that the sum of all terms whose component in
$\fncell G {n-j-1} c {i-j}$ equal to $s^\tau_j$ is precisely $\delta(t_j^\tau
\otimes g)$.
By \autoref{lemma:g-stable-k-complex},
we find that for each $j$, $t_j^\tau = \delta(r_j^\tau \otimes g)$, for some
$r_j^\tau \in \fncell G {j} c {j-1}$. Indeed this is because the
$\mathcal{K}$-complex, $\mathcal K(k, A,  k[g,g^{-1}])$ is exactly dual to the
complex whose $j^{th}$ cohomology group vanishing shows that there is no
obstruction to finding such an $r_j^{\tau}$.
Now, consider the element
$\sum_{j=0}^i \sum_\tau s_j^\tau \otimes r_j^\tau \otimes g \in \fncell G n c {i-1}$.
Applying the coboundary map $\delta$ to this, and using that $\delta(r_j^\tau
\otimes g) = t_j^\tau$, we obtain

\begin{align*}
\delta(\sum_{j=0}^i \sum_\tau s_j^\tau \otimes r_j^\tau \otimes g) &=
\sum_{j=0}^i \sum_\tau \delta(s_j^\tau \otimes r_j^\tau) \otimes g + (-1)^{n-i}
\sum_{j=0}^i \sum_\tau s_j^\tau \otimes
t_j^\tau \\
&=
\sum_{j=0}^i \sum_\tau \delta(s_j^\tau \otimes r_j^\tau) \otimes g + (-1)^{n-i}y.
\end{align*}

We conclude the statement by taking 
\begin{align*}
	z &:= (-1)^{n-i+1}\sum_{j=0}^i \sum_\tau
	\delta(s_j^\tau \otimes r_j^\tau), \\
w &:= (-1)^{n-i}\sum_{j=0}^i \sum_\tau s_j^\tau \otimes r_j^\tau.
\qedhere
\end{align*}
\end{proof}

\begin{lemma}
	\label{lemma:iterated-more-g}
	Suppose $n \geq 0, i \geq 0$ and 
$x \in H^i(\cphurc G {n} c, k)$ is represented by a cocycle of the form
$y_1 \otimes g$ for some $y_1 \in \fncell G {n-1} c i$.
Then $x$ is cohomologous to $0$.
\end{lemma}
\begin{proof}
Applying \autoref{lemma:one-more-g} yields 
\begin{align*}
	y_1 \otimes g = \delta(w_1
\otimes g) \otimes g + y_2 \otimes g \otimes g = \delta(w_1 \otimes g \otimes g)
+ y_2 \otimes g \otimes g.
\end{align*}
Hence, $y_1 \otimes g$ is cohomologous $y_2 \otimes g \otimes g$, where, $y_2 \in \cfncell G
{n-2} c i$.
Applying \autoref{lemma:one-more-g} iteratively, we inductively find 
that for any $1 \leq j \leq n$, 
there is $y_j \in \cfncell G {n-j} c i$ for which
\begin{align*}
	y_{j-1} \otimes \underbrace{g \otimes \cdots \otimes g}_{j-1 \text{ times }}
	&= \delta(w_j \otimes g) \otimes \underbrace{g \otimes \cdots \otimes
	g}_{j-1 \text{ times }}
	+ y_{j} \otimes \underbrace{g \otimes \cdots \otimes g}_{j \text{ times
	}} \\
	&=
\delta(w_j \otimes g \otimes \underbrace{g \otimes \cdots \otimes
g}_{j-1 \text{ times }})
	+ y_{j} \otimes \underbrace{g \otimes \cdots \otimes g}_{j \text{ times
	}}.
\end{align*}
Hence, 
$y_j \otimes \underbrace{g \otimes \cdots \otimes g}_{j \text{ times }}$
is cohomologous to $x$. When $j = n$, we find that 
$x$
is cohomologous to a multiple of $\underbrace{g \otimes g \otimes \cdots \otimes
g}_{n \text{
times}}$, and hence cohomologous to $0$ because $x \in H^i(\cphurc G {n} c,
k)$.
\end{proof}

As mentioned prior to \autoref{lemma:one-more-g}
our next goal will be to show that $z$ is cohomologous to a cocycle in a form
ending in $g$.
We now begin preparations to accomplish this in
\autoref{proposition:cocycle-form}.

\begin{notation}
\label{notation:cells}
We will use the notation 
$\cfncell G n c i$ for the subspace of $\fncell G n c i$ spanned by those basis
elements 
such that the union of $g$ with those elements of $c$ appearing in that basis
element generate all of $G$.
More precisely, $\cfncell G n c i$ is generated by tensors of words of the form
$w_1\otimes \cdots \otimes w_{n-i}$
with $w_j = g_{j,1} \cdots g_{j,v_j}$ and the $g_{s,t}$ for $1 \leq s \leq n-i$
and $1 \leq t \leq v_s$ all together with $g$ generate $G$.
Note that $\cfncell G n c \bullet$ forms a subcomplex of the chain complex
$\fncell G n c \bullet$.
We will use $\fncocycle G n c i$ for the subspace of $\fncell G n c i$ consisting
of cocycles, $\fncoboundary G n c i$ for the subspace of $\fncell G n c i$
consisting of coboundaries, and $\fncohomology Gn c i := \fncocycle Gn c i/
\fncoboundary G n c i.$
\end{notation}

\begin{notation}
\label{notation:z}
Suppose we are in the situation of
\autoref{proposition:stabilize-by-one-element-trivial}, 
In particular, 
$z [g]^j[h] = 0$ for any $j \geq 0$ and $h \neq g$.
Choose some sufficiently large $n$ so
that we may represent $x$ by a class $z \in \cfncell G {n} c i$.
(How large we have to take $n$ will be determined in the proof of
\autoref{proposition:cocycle-form}.)
\end{notation}

As previously mentioned, our aim will be to prove \autoref{proposition:stabilize-by-one-element-trivial},
which amounts to showing $z =0$, after modification by a coboundary.
The next lemma translates the hypothesis that each $[g]^j[h]$, for $h \in c -
g$ and $j \geq 0$, kills $z$ to a
concrete description of the form of $z$.

\begin{lemma}
\label{lemma:ending-form}
	With notation as in \autoref{notation:cells} and \autoref{notation:z},
	for any fixed $m$, and every $s < m$, any cocycle $z \in \fncell G n c i$ is
	cohomologous to an element whose projection onto
	$\fncell G {n-s} c i \otimes \underbrace{k\{c\} \otimes \cdots
	\otimes k\{c\}}_{s \text{ times}}$ is of the form
	$z_s \otimes \underbrace{g \otimes \cdots \otimes g}_{s \text{ times}}$.
\end{lemma}
\begin{proof}
	We first claim that, after modifying $z$ by a coboundary, we may assume 
	$z[g]^s[h]$ for any fixed value of $s$.
	To see this, note that for any $h \neq g$, we know $z[g]^s[h]$ is a coboundary by \autoref{notation:z}. 
	If $z[g]^s[h] = \delta(x)$, we then find that $z - \delta(x \otimes h \otimes
	\underbrace{g
\otimes \cdots \otimes g}_{\text{ $s$ times}}) = 0$.
Therefore, we may modify $z$ by a coboundary to assume that $z[g]^s[h]= 0$ for
all $h \neq g$ and any fixed value of $s$.

	Applying the above with $s = 1$ shows the lemma statement holds for $s = 1$.
	By induction on $s$, assuming this holds for $s - 1$ we may assume the
	projection onto 
	$\fncell G {n-(s-1)} c i \otimes \underbrace{k\{c\} \otimes \cdots
	\otimes k\{c\}}_{s-1 \text{ times}}$ is of the form
	$z_{s-1} \otimes \underbrace{g \otimes \cdots \otimes g}_{s \text{ times}}$.
	Then, again modifying $z$ by a coboundary, we may assume $[g]^s[h]$ acts by $0$ on $z$. We find $z_{s-1}$
	has projection onto $\fncell G {n-s} c i \otimes k\{c\}$ with
	$k\{c\}$ term in the span of $g$. Hence, the projection onto
	$\fncell G {n-s} c i \otimes \underbrace{k\{c\} \otimes \cdots
	\otimes k\{c\}}_{s \text{ times}}$ is of the form
	$z_s \otimes \underbrace{g \otimes \cdots \otimes g}_{s \text{ times}}$.
\end{proof}

We next wish to show that $z$ is cohomologous to a cochain in a form to
which we can apply \autoref{lemma:ending-form}.
The next lemma puts a serious constraint on what $z$ can look like.

\begin{lemma}
\label{claim:vanishing}
Keeping notation as in
\autoref{notation:cells} and \autoref{notation:z},
for $2 \leq j\leq t \leq i+1$, 
$z$ is cohomologous to a cocycle whose projection to
\begin{equation}
\label{equation:end-form}
\oplus_{\alpha=0}^{i+1-j} \fncell G {n-j-\alpha} c {i-j+1} \otimes k\{c^{j} \} \otimes
\underbrace{k\{c\}
\cdots \otimes k\{c\}}_{\alpha \text{ times}}
\end{equation}
is zero. 

Moreover, the projection onto 
$\fncell G {n-s} c i \otimes \underbrace{k\{c\} \otimes \cdots
\otimes k\{c\}}_{s \text{ times}}$ is of the form
$z_s \otimes \underbrace{g \otimes \cdots \otimes g}_{s \text{ times}}$
for $s \leq i+2$.
\end{lemma}
We will prove \autoref{claim:vanishing} later in \autoref{subsubsection:proof-vanishing} after
establishing some preliminary lemmas.
Before proving this, let us see why this implies the following proposition,
which lets us write $z$ with a $g$ on the right.

\begin{proposition}
\label{proposition:cocycle-form}
With notation as in \autoref{notation:cells} and \autoref{notation:z},
$z$ is cohomologous to a cocycle of the form $y \otimes g$.
\end{proposition}
\begin{proof}
For $t = i + 1$, after modifying $z$ by a
coboundary, we may assume it satisfies the conclusion of
\autoref{claim:vanishing}.
We find $z$ has zero projection
to
$\fncell G {n-j} c {i-j+1} \otimes k\{c^j \}$ for all $2 \leq j \leq i + 1$.
This implies
that 
$z \in \fncell G {n-1} c i \otimes k\{c\}$. The second condition 
of \autoref{claim:vanishing}
with $s =
1$ implies
that $z = y \otimes g$ for some $y \in \fncell G {n-1} c i$, as we wished to show.
\end{proof}

Summarizing, what we have accomplished so far, in order to prove
\autoref{proposition:stabilize-by-one-element-trivial}, it remains to prove
\autoref{claim:vanishing};
to see exactly why this suffices, one may examine the proof of 
\autoref{proposition:stabilize-by-one-element-trivial},
given in \autoref{subsubsection:proof-of-stabilize-by-one}.
This is probably the most
involved proof in the paper, and so we will require a number of sublemmas.
We first show certain projections of $z$ have parts which are cocycles.

\begin{lemma}
	\label{lemma:cocycle-modification}
	Using notation as in \autoref{claim:vanishing}, suppose
	\autoref{claim:vanishing} holds for $t-1$. (This is vacuous if $t=2$.)
	Then, for $2 \leq j \leq t$, we may modify $z$ by a coboundary so that the conclusion of \autoref{claim:vanishing} still holds for $t-1$ and the projection
	of $z$
	onto
\begin{equation}
\label{equation:end-form-augmented}
\oplus_{\alpha=0}^{i+2-j} \fncell G {n-j-\alpha} c {i-j+1} \otimes k\{c^{j} \} \otimes
\underbrace{k\{c\}
\cdots \otimes k\{c\}}_{\alpha \text{ times}}
\end{equation}
	in fact lies in
\begin{equation}
\label{equation:cycle-end-form}
\oplus_{\alpha=0}^{i+2-j} \fncocycle G {n-j-\alpha} c {i-j+1} \otimes k\{c^{j} \} \otimes
\underbrace{k\{c\}
\cdots \otimes k\{c\}}_{\alpha \text{ times}}.
\end{equation}
\end{lemma}
We note that \eqref{equation:end-form-augmented} is very similar to
\eqref{equation:end-form} except that the index on $\alpha$ goes up to $i + 2 -
j$ instead of only $i + 1 - j$.
\begin{proof}
	To see this, let us show the projection $\on{pr}_j(z)$ of $z$ onto the summand 
of \eqref{equation:end-form-augmented} indexed by $\alpha$ lies in 
$\fncocycle G {n-j-\alpha} c {i-j+1} \otimes k\{c^{j} \} \otimes
\underbrace{k\{c\}
\cdots \otimes k\{c\}}_{\alpha \text{ times}}
,$ after modifying $z$ by a coboundary.
If it did not, there would necessarily be a term
in $\delta(\on{pr}_j(z))$ of the
form $\fncoboundary G {n-j-\alpha} c {i-j+1} \otimes k\{c^j\}
\otimes \underbrace{k\{c\}
\cdots \otimes k\{c\}}_{\alpha \text{ times}}$. 
Hence, since $\delta(z) = 0$ so some term must cancel the above term
in $\delta(\on{pr}_j(z))$, there would be a term in $z$ of the form
$\fncell G {n-j-\alpha} c {i-j+1} \otimes k\{c^v\} \otimes k\{c^{j-v}\}
\otimes \underbrace{k\{c\}
\cdots \otimes k\{c\}}_{\alpha \text{ times}}$,
(with $v > 0$ and $j - v > 0$)
whose coboundary via shuffling the the words of length $v$ and $j-v$ is
nonzero.
Now, if $j > 2$,
we can modify $z$ by a coboundary so that
this is impossible by our hypothesis for \autoref{claim:vanishing} (which we are
assuming holds for $t - 1$) with $j$
replaced by $j-v$, as no such terms exist.
(Note that because $j - v < j$ we obtain that
$i - (j-v) + 1 \geq i-j+2$, so we obtain this when $\alpha = i + 2 -j$ as well.)
Finally, if $j = 2$, we must have $v = 1$ and $j - v = 1$.
Note that $\alpha \leq i + 2 - j = i$, so $\alpha + 2 \leq i + 2$.
Hence, we may apply the second part of 
\autoref{claim:vanishing}
for $s = \alpha+2$, (using that $s$ is at most $i +2$ as explained above,) such terms are of the form $z_{\alpha+2} \otimes \underbrace{g \otimes
\cdots \otimes g}_{\alpha+2 \text{ times}}$, and the
coboundary of $\underbrace{g \otimes
\cdots \otimes g}_{\alpha+2 \text{ times}}$ vanishes.
\end{proof}

We next wish to show the cocycles appearing in the statement of
\autoref{lemma:cocycle-modification} can be expressed in terms of particular
representatives, pulled back from the cohomology of configuration space.
We set up notation for these representatives
and then show the cohomology can be expressed in terms of these representatives
in \autoref{lemma:projection-representative}.

\begin{notation}
	\label{notation:conf-cohomology-generators}
	We keep notation as in \autoref{notation:cells} and
	\autoref{notation:z}.
	Choose $n$
large enough so that $n - t - 1$ is in the stable
range for the $(i-t+1)$st cohomology and such that \autoref{claim:vanishing} is
true for $t-1$.  
In particular, we are assuming via \autoref{notation:z} that
$\stableboundary j G
	c g \to \stable j {\id} \beta {\id}$ is an isomorphism
	for any $j < i$ and any $g \in G$.
	Hence, the stable cohomology classes in 
$(\stable {i-t+1} {G} {\beta} {c})^\vee$
are all pulled back from configuration space.

Since configuration space corresponds to a Hurwitz space for the trivial group
$\id$, we use 
$(\stable {i-t+1}
{\id} {\beta} {\id})^\vee$ as notation for the stable cohomology of configuration
space.
For all $n > 1$ choose a sequence of cocycles $w_n \in \fncocycle {\id} n {\id}
{i-t+1}$ compatible with the costabilization map, which are $0$ whenever 
$(\stable {i-t+1}
{\id} {\beta} {\id})^\vee = 0$,
and otherwise project to a nonzero element in 
$(\stable {i-t+1}
{\id} {\beta} {\id})^\vee $.
For the readers benefit, we point out that  $(\stable {i}
{\id} {\beta} {\id})^\vee = 0$ unless $i=0,1$, in which case it is $1$-dimensional.
Fixing a value of $n \bmod 2$,
if we let $f:  \boundarycphurc G n c g \to \cphurc {\id} n {\id}$ denote the
projection from the fixed stable component with boundary monodromy $g \in G$,
we obtain a sequence of cocycles $f^*(w_n) \in \fncocycle G n c {i-t+1}$
compatible with the costabilization maps uniquely
representing a spanning cocycle for $H^{l-i+1}(\boundarycphurc G n c g,k),$ for $n$ in the stable range.
\end{notation}

\begin{lemma}
	\label{lemma:projection-representative}
	Use notation from \autoref{notation:cells}, \autoref{notation:z}, and
	\autoref{notation:conf-cohomology-generators}.
	Assume \autoref{claim:vanishing} holds for $t - 1$.
(This is vacuous if $t=2$.)
	For $2 \leq j \leq t$, we can modify $z$ 
by coboundaries so
	its projection to 
	\eqref{equation:end-form-augmented} lies in 
\eqref{equation:cycle-end-form}.
Moreover, when this projection is expanded as a  of simple tensors in terms of a basis
for
$\fncocycle G {n-j-\alpha} c {i-j+1}$,
we may assume each such basis element
has component in
$\fncocycle G {n-j-\alpha} c {i-j+1}$
equal to a multiple of the chosen representative $f^*(w_{n-j-\alpha})$.
\end{lemma}
\begin{proof}
	Since we are assuming \autoref{claim:vanishing} holds for $t - 1$, we
	can apply
	\autoref{lemma:cocycle-modification}. We can then assume the projection
	of $z$ to \eqref{equation:end-form-augmented} lies in
	\eqref{equation:cycle-end-form}.
	Now, modifying the resulting element by a coboundary of an element in 
	$\oplus_{\alpha=0}^{i+2-j} \fncell G {n-j-\alpha} c {i-j} \otimes k\{c^{j} \} \otimes
	\underbrace{k\{c\}
		\cdots \otimes k\{c\}}_{\alpha \text{ times}}$
	so that 
	we may assume each such term
has component in
$\fncocycle G {n-j-\alpha} c {i-j+1}$
equal to a multiple of
$f^*(w_{n-j-\alpha})$, as desired.
\end{proof}

The next lemma further constrains the form of $z$ by using the above constructed
representatives to show that many projections of $z$ vanish.
	
\begin{lemma}
	\label{lemma:vanishing-projection}
	Using notation as in \autoref{claim:vanishing}, suppose
	\autoref{claim:vanishing} holds for $t-1$.
(This is vacuous if $t=2$.)
	Assume $z$ has been modified by a coboundary to satisfy the conclusion of
	\autoref{lemma:projection-representative}.
Suppose $w \in \fncell G n c i$ is a basis vector in the form \eqref{equation:w-form},
and $w$ does not lie in
\begin{align}
\label{equation:delta-dependence-no-cocycle}
\fncell G {n-t-\alpha} c {i-t+1} \otimes k\{c^{t} \} \otimes
\underbrace{k\{c\}
\cdots \otimes k\{c\}}_{\alpha \text{ times}}
\oplus
\fncell G {n-t-\alpha-1} c {i-t+1} \otimes k\{c^{t} \} \otimes k\{c\} \otimes
\underbrace{k\{c\}
\cdots \otimes k\{c\}}_{\alpha \text{ times}}.
\end{align}
Then, for
$0 \leq \alpha \leq i + 1 - t$,
$\on{pr}_w(z)$ the projection of $z$ onto any such $w$, 
the projection of $\delta\left( \on{pr}_w(z) \right)$
to
\begin{equation}
	\fncell G {n-t-1
	- \alpha}
	c {i - t+1}
	\otimes k\{c^{t+1}\} \otimes
\underbrace{k\{c\} \cdots \otimes k\{c\}}_{\alpha \text{ times}}
\label{equation:boundary-projection-term}
\end{equation}
vanishes.
Moreover, the projection of $z$ onto
\eqref{equation:delta-dependence-no-cocycle} lands in 
\begin{align}
\label{equation:delta-dependence}
\fncocycle G {n-t-\alpha} c {i-t+1} \otimes k\{c^{t} \} \otimes
\underbrace{k\{c\}
\cdots \otimes k\{c\}}_{\alpha \text{ times}}
\oplus
\fncocycle G {n-t-\alpha-1} c {i-t+1} \otimes k\{c^{t} \} \otimes k\{c\} \otimes
\underbrace{k\{c\}
\cdots \otimes k\{c\}}_{\alpha \text{ times}}.
\end{align}
%
\end{lemma}
\begin{proof}
	First, the final statement that the projection onto
	\eqref{equation:delta-dependence-no-cocycle}
	factors through
	\eqref{equation:delta-dependence} follows from
	\autoref{lemma:cocycle-modification}.
	This importantly uses that $\alpha$ appearing in the sum in
	\eqref{equation:end-form-augmented} runs up to $i + 2 - j$ while
	$\alpha$ appearing in \eqref{equation:end-form} only runs up to $i + 1 -
	j$, so that \autoref{lemma:cocycle-modification} applies to both terms in in
	\eqref{equation:delta-dependence-no-cocycle}.

	Using our inductive hypothesis that \autoref{claim:vanishing} holds for
	$j$ with $2 \leq j < t$,
	the projection of $z$ onto all terms of the form 
\begin{equation}
\fncell G {n-t-\alpha} c {i-t+1} \otimes k\{c^v\} \otimes k\{c^{t-v} \} \otimes
\underbrace{k\{c\}
\cdots \otimes k\{c\}}_{\alpha \text{ times}}
\end{equation}
with $0 \leq \alpha \leq i + 1 - t$, $v \geq 1$, and $t-v \geq 2$
vanishes. 
(This uses that the above condition holds when $\alpha \leq i + 1 - (t-v)$, so
	it holds in particular when
$\alpha \leq i + 2 - t$.)
The only other term whose coboundary can contribute to the projection
onto \eqref{equation:boundary-projection-term}
is
\begin{align}
\label{equation:longer-end}
\fncell G {n-t-\alpha-1} c {i-t} \otimes k\{c^{t+1}\} \otimes
\underbrace{k\{c\}
\cdots \otimes k\{c\}}_{\alpha \text{ times}}.
\end{align}
However, the contribution from the coboundary of this term to 
\eqref{equation:boundary-projection-term}
necessarily lies in 
\begin{equation}
\label{equation:boundary-of-edge-term}
	\fncoboundary G {n-t-1 - \alpha} c {i - t+1}\otimes k\{c^{t+1}\} \otimes
\underbrace{k\{c\} \cdots \otimes k\{c\}}_{\alpha \text{ times}}.
\end{equation}
On the other hand, using notation from
\autoref{notation:conf-cohomology-generators},
each simple tensor in the expansion of the image of the
projection of $z$ to \eqref{equation:delta-dependence} along the 
$\fncocycle G {n-t-\alpha-1} c {i-t+1}$ factor consists of a cocycle that is a multiple of $f^*(w_{n-j-\alpha})$. 

Since such a cocycle is not a coboundary unless it is $0$, the sum
of the terms
lying in the subspace \eqref{equation:boundary-of-edge-term}
must vanish. In other words,
the projection of $z$ to \eqref{equation:longer-end}
in fact lies in 
$\fncocycle G {n-t-\alpha-1} c {i-t} \otimes k\{c^{t+1}\} \otimes
\underbrace{k\{c\}
\cdots \otimes k\{c\}}_{\alpha \text{ times}}.
$
This establishes our claim.
\end{proof}

The above lemmas are relatively straightforward reductions that allow us to
significantly simplify the form of $z$.
The next lemma is really the key step,
which uses the two-sided $\mathcal K$-complex introduced in the previous
section.
Exactness of that complex gives us exactness in the next lemma, which we will
then use to further simplify the form of $z$.
	
\begin{lemma}
	\label{lemma:exact-triangle}
	Let $\alpha\geq 0, t\geq 2$ and $n\geq 0$ be integers with $\alpha + t <
	n$.
	Continuing to use notation as in \autoref{notation:cells} and
	\autoref{notation:z}, the restriction of the following three term
	sequence to the
subcomplex spanned by tensors of elements which generate
$G$ is exact:
\begin{equation}
\label{equation:2-sided-extraction}
\begin{aligned}
&\fncohomology G {n-t+1-\alpha} c {i-t+1}  \otimes k\{c^{t-1}\} \otimes \underbrace{k\{g\} \otimes
\cdots k\{g\}}_{\alpha \text{ times}}
\\
&\oplus
\fncohomology G {n-t-\alpha} c {i-t+1} \otimes k\{c^{t-1}\} \otimes k\{g\}
\otimes \underbrace{k\{g\} \otimes
\cdots k\{g\}}_{\alpha \text{ times}}
\\
&\oplus \fncohomology G {n-t-1-\alpha} c {i-t+1} \otimes k\{c^{t-1}\}  \otimes k\{g\} \otimes
k\{g\} \otimes \underbrace{k\{g\} \otimes
\cdots k\{g\}}_{\alpha \text{ times}}
&\\
\xrightarrow{\mu_\alpha} &\fncohomology G {n-t-\alpha} c {i-t+1} \otimes k\{c^t\}
\otimes \underbrace{k\{g\} \otimes
\cdots k\{g\}}_{\alpha \text{ times}}
\\
&\oplus \fncohomology G
{n-t-1-\alpha} c {i-t+1} \otimes k\{c^t\} \otimes k\{g\} \otimes \underbrace{k\{g\} \otimes
\cdots k\{g\}}_{\alpha \text{ times}}
\\
\xrightarrow{\nu_\alpha} &\fncohomology G {n-t-1-\alpha} c{i-t+1} \otimes k\{c^{t+1}\}
\otimes \underbrace{k\{g\} \otimes
\cdots k\{g\}}_{\alpha \text{ times}}.
\end{aligned}
\end{equation}
\end{lemma}
\begin{proof}
First, if $i-t+1 \notin \{0,1\}$, we are done, as this implies the above complex
is $0$, using that we are inductively assuming 
(via \autoref{notation:z} via
\autoref{proposition:stabilize-by-one-element-trivial})
the $(i-t+1)$th cohomology is
pulled back from configuration space, and this vanishes in degrees more than
$1$.

Hence, we now assume $i - t + 1 \in \{0,1\}$.
In this case, since 
$\fncohomology G {n-t-1-\alpha} c {i-t+1}$ is pulled back from configuration
space, the module structure in the case $i - t + 1 = 1$ is the same as the
module structure in the case $i - t + 1 = 0$, so we may assume $i - t + 1 = 0$.
Let $\widetilde{\fncohomology G {n-t-1-\alpha} c{i-t+1}} \subset\fncohomology G {n-t-1-\alpha} c{i-t+1}$
denote the codimension $1$ subspaces spanned by all components other than the
one with only monodromy $g$.
There is a subcomplex
of \eqref{equation:2-sided-extraction} given by replacing each 
$\fncohomology G {\bullet} c {\bullet}$ term with the corresponding codimension
$1$ subspace
$\widetilde{\fncohomology G {\bullet} c {\bullet}}$.
Specifically, this subcomplex is given by
\begin{equation}
\label{equation:2-sided-extraction-subcomplex}
\begin{aligned}
	&\widetilde{\fncohomology G {n-t+1-\alpha} c {i-t+1}}  \otimes k\{c^{t-1}\} \otimes \underbrace{k\{g\} \otimes
\cdots k\{g\}}_{\alpha \text{ times}}
\\
&\oplus
\widetilde{
\fncohomology G {n-t-\alpha} c {i-t+1}} \otimes k\{c^{t-1}\} \otimes k\{g\}
\otimes \underbrace{k\{g\} \otimes
\cdots k\{g\}}_{\alpha \text{ times}}
\\
&\oplus \widetilde{\fncohomology G {n-t-1-\alpha} c {i-t+1}} \otimes k\{c^{t-1}\}  \otimes k\{g\} \otimes
k\{g\} \otimes \underbrace{k\{g\} \otimes
\cdots k\{g\}}_{\alpha \text{ times}}
&\\
\xrightarrow{\mu_\alpha} &\widetilde{\fncohomology G {n-t-\alpha} c {i-t+1}} \otimes k\{c^t\}
\otimes \underbrace{k\{g\} \otimes
\cdots k\{g\}}_{\alpha \text{ times}}
\\
&\oplus \widetilde{\fncohomology G
{n-t-1-\alpha} c {i-t+1}} \otimes k\{c^t\} \otimes k\{g\} \otimes \underbrace{k\{g\} \otimes
\cdots k\{g\}}_{\alpha \text{ times}}
\\
\xrightarrow{\nu_\alpha} &\widetilde{\fncohomology G {n-t-1-\alpha} c{i-t+1}} \otimes k\{c^{t+1}\}
\otimes \underbrace{k\{g\} \otimes
\cdots k\{g\}}_{\alpha \text{ times}}.
\end{aligned}
\end{equation}
This subcomplex is in fact exact using \autoref{lemma:zero-and-iso-vanishes}, as
we next explain.
The reason this is exact is that it can be identified, via the spectral sequence
described in 
\autoref{subsection:two-sided-k-complex}
as dual to part of the cohomology of the complex
$\widetilde{\fncohomology G {\bullet} c{i-t+1}} \otimes_A k[g,g^{-1}]$.
(This identification is quite similar to the description of $C_{\bullet,
\bullet}$ in \autoref{lemma:two-sided-k-complex-descrption}
and we omit further details.)
The module 
$\widetilde{\fncohomology G {\bullet} c{i-t+1}}$
is then spanned by modules on which some $h
\in c, h \neq g$ acts invertibly, and hence it follows from 
\autoref{lemma:zero-and-iso-vanishes} that this subcomplex is exact.
Now, let $\overline{\fncohomology G {n-t-1-\alpha} c{i-t+1}} := \fncohomology G
{n-t-1-\alpha} c{i-t+1}/ \widetilde{\fncohomology G
{n-t-1-\alpha} c{i-t+1}}$ denote the $1$-dimensional quotient.
Exactness of \eqref{equation:2-sided-extraction-subcomplex} implies that the
cohomology of \eqref{equation:2-sided-extraction} is identified with the
cohomology of the quotient complex, which is explicitly given by
\begin{equation}
\label{equation:2-sided-extraction-quotient}
\begin{aligned}
	&\overline{\fncohomology G {n-t+1-\alpha} c {i-t+1}}  \otimes k\{c^{t-1}\} \otimes \underbrace{k\{g\} \otimes
\cdots k\{g\}}_{\alpha \text{ times}}
\\
&\oplus
\overline{
\fncohomology G {n-t-\alpha} c {i-t+1}} \otimes k\{c^{t-1}\} \otimes k\{g\}
\otimes \underbrace{k\{g\} \otimes
\cdots k\{g\}}_{\alpha \text{ times}}
\\
&\oplus \overline{\fncohomology G {n-t-1-\alpha} c {i-t+1}} \otimes k\{c^{t-1}\}  \otimes k\{g\} \otimes
k\{g\} \otimes \underbrace{k\{g\} \otimes
\cdots k\{g\}}_{\alpha \text{ times}}
&\\
\xrightarrow{\mu_\alpha} &\overline{\fncohomology G {n-t-\alpha} c {i-t+1}} \otimes k\{c^t\}
\otimes \underbrace{k\{g\} \otimes
\cdots k\{g\}}_{\alpha \text{ times}}
\\
&\oplus \overline{\fncohomology G
{n-t-1-\alpha} c {i-t+1}} \otimes k\{c^t\} \otimes k\{g\} \otimes \underbrace{k\{g\} \otimes
\cdots k\{g\}}_{\alpha \text{ times}}
\\
\xrightarrow{\nu_\alpha} &\overline{\fncohomology G {n-t-1-\alpha} c{i-t+1}} \otimes k\{c^{t+1}\}
\otimes \underbrace{k\{g\} \otimes
\cdots k\{g\}}_{\alpha \text{ times}}.
\end{aligned}
\end{equation}
To conclude, we wish to show the part of this complex spanned by tensors of elements generating $G$
is exact.
This part of the quotient complex identified as dual to the subcomplex
$D^{-t-1}_{\bullet, \bullet}$ defined in
\autoref{proposition:upper-triangle-exact}.
Hence, it is exact by \autoref{proposition:upper-triangle-exact}.
\end{proof}

Combining the above lemmas, we now prove \autoref{claim:vanishing}.

\subsubsection{Proof of \autoref{claim:vanishing}}
\label{subsubsection:proof-vanishing}
Note that the second part holds by \autoref{lemma:ending-form}, so
it remains to prove the first part, which we will do
by induction on $t$. We assume 
\autoref{claim:vanishing} holds for $t -1$.
(When $t = 2$, this condition is vacuous.)

After modifying $z$ by a coboundary,
we may assume the projection of $z$ to \eqref{equation:delta-dependence} 
lies in the kernel of $\nu_\alpha$, defined
in \eqref{equation:2-sided-extraction},
by \autoref{lemma:vanishing-projection}.

We next claim we can modify $z$ by an element in the image of $\mu_\alpha$ so that the projection of
$z$ onto 
$\fncohomology G {n-t-\alpha} c {i-t+1} \otimes k\{c^t\}\otimes \underbrace{k\{g\} \otimes
\cdots k\{g\}}_{\alpha \text{ times}}
$ vanishes.
First, when $\alpha = 0$,
this follows from exactness of \eqref{equation:2-sided-extraction}, established in
\autoref{lemma:exact-triangle}. 
Applying this for each $\alpha$ with $0 \leq \alpha \leq i - t$,
we can arrange that 
the projection of $z$ onto 
$\fncohomology G {n-t-\alpha} c {i-t+1} \otimes k\{c^t\}\otimes \underbrace{k\{g\} \otimes
\cdots k\{g\}}_{\alpha \text{ times}}
$ vanishes. Note we use here that the cocycles $f^*(w_{n})$ are closed under the
costabilization map.

Finally, using \autoref{lemma:projection-representative},
after modifying $z$ by a coboundary, because its projection to 
$\fncohomology G {n-t-\alpha} c {i-t+1} \otimes k\{c^t\}\otimes \underbrace{k\{g\} \otimes
\cdots k\{g\}}_{\alpha \text{ times}}
$ vanishes, we may also
assume its projection to 
\eqref{equation:end-form} vanishes. 
\qed

\subsubsection{Proof of \autoref{proposition:stabilize-by-one-element-trivial}}
\label{subsubsection:proof-of-stabilize-by-one}

By \autoref{proposition:cocycle-form},
we may assume that, after modifying $z$ by a coboundary, $z = y_1 \otimes g$ for
$y_1 \in \fncell G {n-1} c i$.
In fact, we have, $y_i \in \cfncell G {n-1} c i$
using the definition of $\cfncell G {n-1} c i$
and the assumption that
$y_1 \otimes g \in \cfncell G {n-1} c i$ using that $z \in \stableboundary i G c
g$.
We now conclude by \autoref{lemma:iterated-more-g}.
\qed

\begin{example}
	\label{example:h1-argument}
	In this example, we run through the argument for
	\autoref{proposition:stabilize-by-one-element-trivial}
	in the case of $i = 1$, i.e., for the first (co)homology group.
	We use notation as in \autoref{notation:cells} and \autoref{notation:z}.
	We fix an element $g \in c$, assume $n$ is sufficiently large, and assume we have an element $\overline{x} \in H^1(\cphurc G {n+1} c, k)$
	such that $\overline{x}[h] = 0$, $\overline{x}[g][h] = 0$, and
	$\overline{x}[g]^2[h] = 0$ for all $h \neq g$.
	We can choose a representative $x \in \fncocycle G n c 1$,
	projecting to $\overline{x}$.
	We can then also express $x$ as an element of
	\begin{align*}
		&\fncell G {n-2} c 0 \otimes k\{c^2\} \\
		&\oplus 
		\fncell G {n-3} c 0 \otimes k\{c^2\} \otimes k\{c\} \\
		&\oplus
		\fncell G {n-4} c 0 \otimes k\{c^2\} \otimes  k\{c\} \otimes
		k\{c\} \\
		&\oplus
		\fncell G {n-3} c 1 \otimes k\{c\} \otimes k\{c\} \otimes k\{c\}.
	\end{align*}
	We picture $x$ on the left of \autoref{figure:h1-fuks-cells}.
\begin{figure}
	\includegraphics[scale=.6]{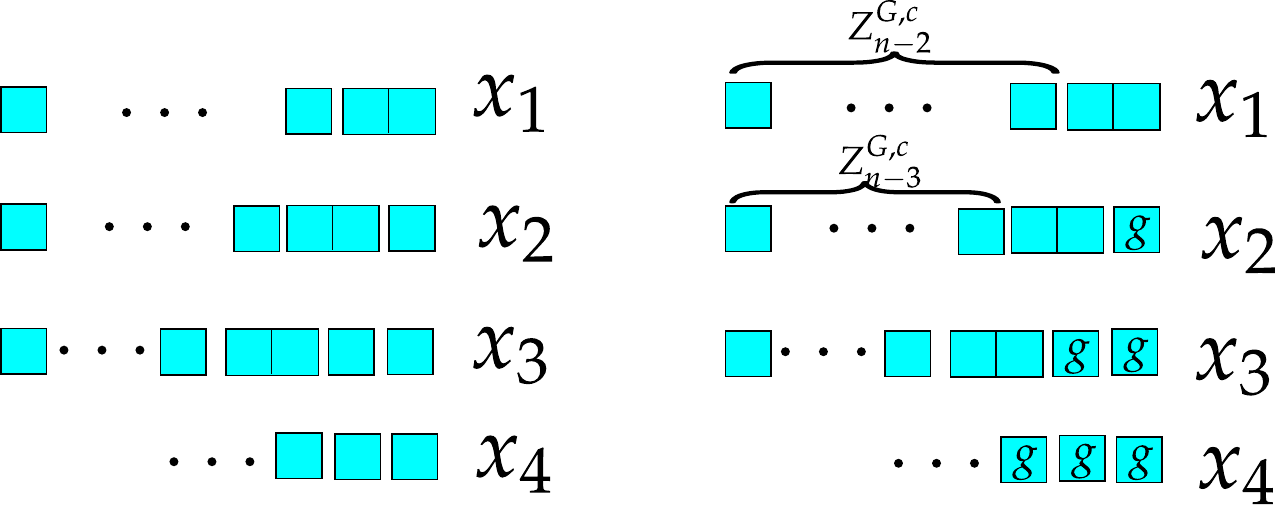}
	\caption{The left is a picture of the Fox-Neuwirth/Fuks cell structure of an
	element of the first stable cohomology. The right hand side pictures
some additional constraints we may impose on the cell structure, as explained in
\autoref{example:h1-argument}.}
\label{figure:h1-fuks-cells}
\end{figure}

We next carry out \autoref{lemma:ending-form} explicitly in this case.
	Write $x = x_1 + x_2 + x_3 + x_4$, where $x_j$ is in the $j$th component
	above.
	We know $\overline{x}[h] = 0$.
	This means $x[h] = \delta(x'_h)$ for some $x'_h \in \fncell G n c 0$.
	Then, when we replace $x$ by $x - \sum_{h \neq g} \delta(x'_h \otimes h)$,
	we find that $x[h]= 0$.
	The condition that $x[h] = 0$ for all $h \neq g$ implies that $x_i = y_i \otimes g$ for $2
	\leq i \leq 4$. 
	By similarly modifying $x$ by a coboundary yet again, we can arrange
	that $x[g][h] =0$ and $x[g]^2[h] = 0$.
	The condition that $x[g][h] = 0$ implies $x_i =
	z_i \otimes g \otimes g$ for $3 \leq i \leq 4$.
	Finally, the condition that $x[g]^2[h] = 0$ implies $x_4 = w_4 \otimes g
	\otimes g \otimes g$.
	Hence, we have that $x$ lies in
	\begin{align*}
		&\fncell G {n-2} c 0 \otimes k\{c^2\} \\
		&\oplus 
		\fncell G {n-3} c 0 \otimes k\{c^2\} \otimes k\{g\} \\
		&\oplus
		\fncell G {n-4} c 1 \otimes k\{c^2\} \otimes  k\{g\} \otimes
		k\{g\} \\
		&\oplus
		\fncell G {n-3} c 1 \otimes k\{g\} \otimes k\{g\} \otimes k\{g\}.
	\end{align*}
	
	We now carry out \autoref{lemma:cocycle-modification} explicitly in this
	case.
	Note that we can write $x_1$ as a sum of tensor products $\sum_j u_j
	\otimes v_j$, where $u_j \in \fncell G {n-2} c 0$ 
	and $v_j \in
	k\{c^2\}$.
	Using that $\delta(x) = 0$,
	we can arrange that each $v_j$ satisfies $\delta(v_j)=0$ because
	$0= \delta(x) = \delta(x_1) + \delta(x_2) + \delta(x_3) + \delta(x_4)$ but
	$\delta(x_2)$, $\delta(x_3)$, and $\delta(x_3)$ all have no terms ending
	in $k\{c^2\}$,
	using that the shuffle of
	$g$ with itself is $0$.
	This shows that $x_1$ lies in 
	$\fncocycle G {n-2} c 0 \otimes k\{c^2\}$.
	Similarly, $x_2 \in \fncocycle G {n-3} c 1 \otimes k\{c^2\} \otimes  k\{g\}$.
	So, altogether, we obtain that $x$ lies in
\begin{align*}
		&\fncocycle G {n-2} c 0 \otimes k\{c^2\} \\
		&\oplus 
		\fncocycle G {n-3} c 0 \otimes k\{c^2\} \otimes k\{g\} \\
		&\oplus
		\fncell G {n-4} c 1 \otimes k\{c^2\} \otimes  k\{g\} \otimes
		k\{g\} \\
		&\oplus
		\fncell G {n-3} c 1 \otimes k\{g\} \otimes k\{g\} \otimes k\{g\}.
	\end{align*}
We picture the present form on $x$ on the right of
	\autoref{figure:h1-fuks-cells}.

	Now, using exactness of the two-sided K-complex, in the form of
	\autoref{proposition:upper-triangle-exact},
	we claim that
	there is some element
	\begin{align*}
		s = s_1 + s_2 + s_3 \in
	\fncohomology G {n-1} c 0 \otimes k\{c\}
	\oplus 
	\fncohomology G {n-2} c 0 \otimes k\{c\} \otimes k\{g\}
	\oplus 
	\fncohomology G {n-3} c 0 \otimes k\{c\} \otimes k\{g\} \otimes k\{g\},
	\end{align*}
	with $s_1 \in \fncohomology G {n-1} c 0 \otimes k\{c\},
	s_2 \in 
	\fncohomology G {n-2} c 0 \otimes k\{c\} \otimes k\{g\}$,
	and
	$s_3 \in \fncohomology G {n-3} c 0 \otimes k\{c\} \otimes k\{g\} \otimes
	k\{g\}$,
	so that the projection of $\delta(s)$ to $\fncell G {n-2} c 0 \otimes
	k\{c^2\}$ agrees with $x_1$.
		This uses that $\fncohomology G {n-3} c 0 = \fncocycle G {n-3} c 0$, as
	there are no coboundaries in the $0$th cohomology.
	We note that to work with the two sided $K$-complex, we really need to
		ignore terms in the subcomplex spanned by $g \otimes \cdots
		\otimes g$, but this is not a problem because our element starts
	in the cohomology of the connected Hurwitz space.
	We picture $s$ in the top row of
	\autoref{figure:two-sided-k-complex-example}.
\begin{figure}
	\includegraphics[scale=.6]{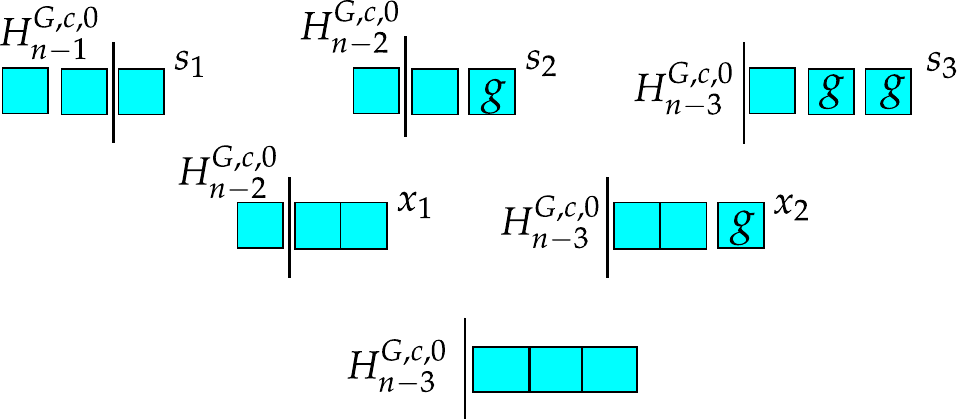}
	\caption{This is a picture of part of the two-sided $K$ complex used to
	compute the stable first cohomology, as relevant to
\autoref{example:h1-argument}.}
\label{figure:two-sided-k-complex-example}
\end{figure}

We now conclude the argument that $x =0$.
	With notation as above, $x - \delta(s)$ lies in $\fncell G {n-1} c 1 \otimes k\{g\}$.
	So, after replacing $x$ by its sum with a coboundary, we can arrange
	that $x \in \fncell G {n-1} c 1 \otimes k\{g\}$.
	Applying \autoref{lemma:one-more-g} 
	shows we may further assume $x \in \fncell G {n-2} c 1 \otimes k\{g\} \otimes k\{g\}$.
	Continuing to apply \autoref{lemma:one-more-g}, we may assume $x$ lies in the
	span of $g \otimes \cdots \otimes g$. Since we are assuming
	$x \in H^1(\cphurc G {n+1} c, k)$, we must have $x =0$.
\end{example}

\section{Computing the stable homology}
\label{section:stable-homology}

In this section, we complete the proof of our main result by computing the stable homology of Hurwitz spaces associated
dihedral groups.
As a preparatory lemma, we show that components with boundary monodromy in $G - c
- \id$ cannot have interesting stable homology.
This argument essentially appears in 
\cite{bianchiM:polynomial-stability}, and via personal communication, we learned
it was also known to Ellenberg, Venkatesh, and Westerland.

\begin{lemma}
	\label{lemma:cyclic-action}
	For $G$ and $c$ as in \autoref{notation:group}
	and $\zeta \in \mathbb Z/\ell \mathbb Z \subset G$ a generator,
	$\ker (\stableboundary i G  c \zeta \to \stable i {\id} 0
	{\id}) = 0$.
\end{lemma}
\begin{proof}
	This is a consequence of \cite[Corollary
	C']{bianchiM:polynomial-stability} and
the proof of \cite[Corollary 5.4]{randal-williams:homology-of-hurwitz-spaces}.
The first identifies $\stableboundary i G  c \zeta$ with the homology of one
component of the group completion of $\coprod_{n \geq 0} \phurc G n c$. The second shows that
the map from the homology of each component of the group completion of
$\coprod_{n \geq 0} \phurc G n c$
to $\stable i {\id} 0{\id}$ is an isomorphism.
	\end{proof}

\begin{lemma}
	\label{lemma:trivial-action}
	For $G$ and $c$ as in \autoref{notation:group},
	for any $g \in c$ and any $h \in c$ with $h \neq g$, $[h]$ acts by $0$
	on $\ker (\stableboundary i G  c g \to \stable i {\id} 0 {\id})$.
	Moreover, $[g]^j[h]$ also acts by $0$ on 
$\ker (\stableboundary i G  c g \to \stable i {\id} 0 {\id})$
for any $j \geq 0$.
\end{lemma}
\begin{proof}
	For $h \neq g$, with $h \in c, g \in c$, we have that $\zeta := gh$ is a
	generator of $\mathbb Z/\ell \mathbb Z \subset G$.
	Therefore, $[h]$ maps 
$\ker (\stableboundary i G  c g \to \stable i {\id} 0 {\id})$
onto
$\ker (\stableboundary i G  c \zeta \to \stable i {\id} 0 {\id})$,
and the latter vanishes by \autoref{lemma:cyclic-action}.

To show $[g]^j [h]$ acts by $0$, simply observe $[g]^j [h] = [h] [h^{-1}gh]^j$.
So this acts by $0$ because $[h]$ acts by $0$.
\end{proof}

Finally, we are prepared to deduce our main result.
\begin{theorem}
	\label{theorem:stable-homology}
	With notation as in \autoref{definition:boundary-monodromy}
	$\stableboundary i G
	c g \to \stable i {\id} \beta {\id}$ is an isomorphism.
\end{theorem}
\begin{proof}
	One can see using transfer maps that the map $\stableboundary i G
	c g \to \stable i {\id} \beta {\id}$ is a split surjection, so it suffices to show that the kernel of this map is $0$.
	First, we assume $g \in c$.
	Consider the stabilization map $\sum_{h \in c} [h]^2$ acting on 
	$\ker(\stableboundary i G c g \to \stable i {\id} 1 {\id})$
	for $g \in c$. 
	Each $[h]^2$ for $h \neq g$ acts by $0$ on this kernel by
	\autoref{lemma:trivial-action}, while
	$[g]^2$ acts by $0$ using
	\autoref{proposition:stabilize-by-one-element-trivial}.
	Note here that we may apply 
	\autoref{proposition:stabilize-by-one-element-trivial} because
	the assumptions of \autoref{notation:z} are satisfied by 
	\autoref{lemma:trivial-action}.
	Overall, this implies 
	$\sum_{h \in c} [h]^2$ acts as $0$ 
	on 
	$\ker(\stableboundary i G c g \to \stable i {\id} 1 {\id})$.
		But
	$\sum_{h \in c} [h]^2$ acts as the identity on this
	space by \autoref{notation:stable-homology}, implying
	$\ker(\stableboundary i G c g \to\stable i {\id} 1 {\id}) =0$.

	The case that $g \notin c$ follows from the case $g \in c$ because each
$[h]$ for $h \in c$ acts as $0$ on 
$\ker(\stableboundary i G  c g \to\stable i {\id} 0 {\id})$, since it sends
$\ker(\stableboundary i G  c g \to\stable i {\id} 0 {\id})$
to $\ker(\stableboundary i G  c {gh} \to\stable i {\id} 1 {\id})$, which we have
just shown to be $0$, and $\sum_{h \in c} [h]^2$ is the identity.
\end{proof}

\bibliographystyle{alpha}
\bibliography{./bibliography}

\def\cprime{$'$} \providecommand{\noopsort}[1]{}
\begin{thebibliography}{EVW16}

\bibitem[Ad69]{arnold:cohomology-of-braid-group}
V.~I. Arnol\cprime~d.
\newblock The cohomology ring of the group of dyed braids.
\newblock {\em Mat. Zametki}, 5:227--231, 1969.

\bibitem[BM23]{bianchiM:polynomial-stability}
Andrea Bianchi and Jeremy Miller.
\newblock Polynomial stability of the homology of hurwitz spaces.
\newblock {\em arXiv preprint arXiv:2303.11194v1}, 2023.

\bibitem[EL23]{ellenbergL:homological-stability-for-generalized-hurwitz-spaces}
Jordan~S Ellenberg and Aaron Landesman.
\newblock Homological stability for generalized hurwitz spaces and selmer
  groups in quadratic twist families over function fields.
\newblock {\em arXiv preprint arXiv:2310.16286v2}, 2023.

\bibitem[ETW17]{ellenbergTW:fox-neuwirth-fuks-cells}
Jordan~S Ellenberg, TriThang Tran, and Craig Westerland.
\newblock Fox-neuwirth-fuks cells, quantum shuffle algebras, and malle's
  conjecture for function fields.
\newblock {\em arXiv preprint arXiv:1701.04541v2}, 2017.

\bibitem[EVW12]{ellenbergVWhomologicalII}
Jordan~S Ellenberg, Akshay Venkatesh, and Craig Westerland.
\newblock Homological stability for hurwitz spaces and the cohen-lenstra
  conjecture over function fields, {II}.
\newblock {\em arXiv preprint arXiv:1212.0923v1}, 2012.

\bibitem[EVW16]{EllenbergVW:cohenLenstra}
Jordan~S. Ellenberg, Akshay Venkatesh, and Craig Westerland.
\newblock Homological stability for {H}urwitz spaces and the {C}ohen-{L}enstra
  conjecture over function fields.
\newblock {\em Ann. of Math. (2)}, 183(3):729--786, 2016.

\bibitem[LL24]{landesmanL:the-stable-homology-of-non-splitting}
Aaron Landesman and Ishan Levy.
\newblock The {C}ohen-{L}enstra moments over function fields via the stable
  homology of non-splitting {H}urwitz spaces.
\newblock {\em In preparation}, 2024.

\bibitem[RW20]{randal-williams:homology-of-hurwitz-spaces}
Oscar Randal-Williams.
\newblock Homology of {H}urwitz spaces and the {C}ohen-{L}enstra heuristic for
  function fields [after {E}llenberg, {V}enkatesh, and {W}esterland].
\newblock {\em Ast\'{e}risque}, (422):Exp. No. 1164, 469--497, 2020.

\bibitem[SW23]{sawinW:conjectures-for-distributions-containing-roots-of-unity}
Will Sawin and Melanie~Matchett Wood.
\newblock Conjectures for distributions of class groups of extensions of number
  fields containing roots of unity.
\newblock {\em arXiv preprint arXiv:2301.00791v1}, 2023.

\bibitem[Woo21]{wood:an-algebraic-lifting-invariant}
Melanie~Matchett Wood.
\newblock An algebraic lifting invariant of {E}llenberg, {V}enkatesh, and
  {W}esterland.
\newblock {\em Res. Math. Sci.}, 8(2):Paper No. 21, 13, 2021.

\end{thebibliography}

\end{document}